\documentclass{amsart}
\usepackage{lipsum}
\makeatletter
\g@addto@macro{\endabstract}{\@setabstract}

\usepackage[dvipsnames]{xcolor}
\usepackage{setspace}
\usepackage[english]{babel}
\usepackage{amsmath,amsfonts,amssymb,amsthm,epsfig,epstopdf,titling,url,array}
\usepackage[margin=1.0in]{geometry}
\usepackage{tikz-cd}

\usepackage{enumerate}
\usepackage{color,hyperref}
\usepackage[noabbrev]{cleveref}
\hypersetup{colorlinks}
\hypersetup{colorlinks={true},linkcolor={red},citecolor=green}
\usepackage{graphicx}

\graphicspath{{images/}}
\theoremstyle{plain}

\newtheorem{theorem}{Theorem}

\newtheorem{thm}{Theorem}[subsection]
\newtheorem{lem}[thm]{Lemma}
\newtheorem{prop}[thm]{Proposition}
\newtheorem{cor}[thm]{Corollary}
\newcommand\scalemath[2]{\scalebox{#1}{\mbox{\ensuremath{\displaystyle #2}}}}
\theoremstyle{plain}
\newtheorem{defn}[thm]{Definition}

\theoremstyle{remark}
\newtheorem{rem}[thm]{Remark}

\newcommand{\mycomment}[1]{}
\usepackage{tikz}
\usepackage{bm} 
\usepackage{mathrsfs}

\usetikzlibrary{decorations.pathmorphing}
\DeclareMathOperator{\GL}{GL}

\DeclareMathOperator{\ch}{ch}

\DeclareMathOperator{\Hom}{Hom}
\usepackage{mathtools}

\DeclareMathOperator{\vol}{vol}

\newcommand{\authorfootnotes}{\renewcommand\thefootnote{\@fnsymbol\c@footnote}}%
\makeatother

\newcommand\blfootnote[1]{%
  \begingroup
  \renewcommand\thefootnote{}\footnote{#1}%
  \addtocounter{footnote}{-1}%
  \endgroup
}

\makeatletter
\renewcommand{\tocsection}[3]{%
  \indentlabel{\@ifnotempty{#2}{\bfseries\ignorespaces#1 #2\quad}}\bfseries#3}
\renewcommand{\tocsubsection}[3]{%
  \indentlabel{\@ifnotempty{#2}{\ignorespaces#1 #2\quad}}#3}

\renewcommand{\tocsubsubsection}[3]{%
  \indentlabel{\@ifnotempty{#2}{\ignorespaces#1 #2\quad}}#3}

\newcommand\@dotsep{4.5}
\def\@tocline#1#2#3#4#5#6#7{\relax
  \ifnum #1>\c@tocdepth 
  \else
    \par \addpenalty\@secpenalty\addvspace{#2}%
    \begingroup \hyphenpenalty\@M
    \@ifempty{#4}{%
      \@tempdima\csname r@tocindent\number#1\endcsname\relax
    }{%
      \@tempdima#4\relax
    }%
    \parindent\z@ \leftskip#3\relax \advance\leftskip\@tempdima\relax
    \rightskip\@pnumwidth plus1em \parfillskip-\@pnumwidth
    #5\leavevmode\hskip-\@tempdima{#6}\nobreak
    \leaders\hbox{$\m@th\mkern \@dotsep mu\hbox{.}\mkern \@dotsep mu$}\hfill
    \nobreak
    \hbox to\@pnumwidth{\@tocpagenum{\ifnum#1=1\bfseries\fi#7}}\par
    \nobreak
    \endgroup
  \fi}
\AtBeginDocument{%
\expandafter\renewcommand\csname r@tocindent0\endcsname{0pt}
}
\def\l@subsection{\@tocline{2}{0pt}{2.5pc}{5pc}{}}
\def\l@subsubsection{\@tocline{2}{0pt}{4pc}{5pc}{}}
\makeatother

\date{} 

\begin{document}

\begin{center}
  \LARGE 
 INTEGRAL STRUCTURES IN SMOOTH $\GL_2(\mathbf{Q}_p)$-REPRESENTATIONS\\
 AND ZETA INTEGRALS

  \normalsize
  \bigskip
  Alexandros Groutides \par 

  Mathematics Institute, University of Warwick\par
\end{center}

\begin{abstract}
Using zeta-integrals and lattices of functions on a spherical variety, we study integral structures in spherical representations of $\GL_2(\mathbf{Q}_p)$ and their interaction with the unique linear functional invariant under an unramified maximal torus. Within this framework, we reformulate and prove the first instance of optimality of abstract integral norm-relations as proposed by Loeffler. We also interpret this as a form of integrality for toric periods associated to modular forms, where part of it can be regarded as an arithmetic integral analogue of Waldspurger's multiplicity one in the unramified setting.
\end{abstract}
\blfootnote{We gratefully acknowledge support from the following research grant: ERC Grant No. 101001051—Shimura varieties and the Birch–Swinnerton-Dyer conjecture.}
\vspace{-2em}
\section{Introduction}
 \subsection{Main results}This paper is devoted to the study of integral structures in spherical Whittaker-type representations of $\GL_2(\mathbf{Q}_p)$ and their interaction with the unique period invariant under the action of an unramified maximal torus $H$. Integral structures and distinction of smooth representations are areas of independent interest. Thus, it is natural to ask how they interact with one another. Motivated by the relative Langlands philosophy of Ben-Zvi–Sakellaridis–Venkatesh \cite{benzvi2024relativelanglandsduality}, we reduce this to studying integral lattices with level structure inside spaces of functions on the spherical variety $H\backslash(\GL_2\times H)$, and equivariance properties with respect to the Hecke algebra. Mapping to the unramified principal-series of $\GL_2(\mathbf{Q}_p)$ we obtain an integral variant of Waldspurger's multiplicity one for unramified toric periods \cite{Waldspurger}, under a certain integrality condition on the spherical Hecke eigensystem. This condition is arithmetic, and naturally satisfied simultaneously at all spherical components of a modular form.

 Additionally, we show that the integral ideal corresponding to input data in the image of the norm map from a ramified level at $p$, coincides with the ideal of the integral Hecke algebra, generated by an ``$L$-factor'', modulo some appropriate integer. This can be regarded as the first instance of integral optimality for abstract norm-relations, as proposed by Loeffler.

Let $p$ be an odd prime. Let $\mathfrak{G}$ be the group $\GL_2\times H$ and $\mathfrak{H}$ the subgroup given by the image of $H$ in $\mathfrak{G}$ under the embedding $h\mapsto (h,h^{-1})$. Let $V$ be any smooth $\mathfrak{G}(\mathbf{Q}_p)$-representation, and 
$$\mathfrak{z}:C_c^\infty(\mathfrak{H}(\mathbf{Q}_p)\backslash\mathfrak{G}(\mathbf{Q}_p),\mathbf{C})\longrightarrow V$$
a $\mathfrak{G}(\mathbf{Q}_p)$-equivariant map. Let $H(\mathbf{Z}_p)[p]$ be the open compact subgroup of $H(\mathbf{Z}_p)$ given by the pullback, under the natural projection $H(\mathbf{Z}_p)\rightarrow H(\mathbf{F}_p)$, of the diagonal $\mathbf{F}_p^\times\subseteq H(\mathbf{F}_p) $. Let $\mathfrak{G}(\mathbf{Z}_p)[p]$ be the ramified open compact level subgroup of $\mathfrak{G}(\mathbf{Z}_p)$ given by $\GL_2(\mathbf{Z}_p)\times H(\mathbf{Z}_p)[p].$ Finally, let $\mathcal{R}^{\mathfrak{G}_p}_{\mathbf{Z}[1/p]}$ be the $\mathbf{Z}[1/p]$-integral spherical Hecke algebra of $(\mathfrak{G}/\GL_1)(\mathbf{Q}_p)$. Our first main result is the following.
 \begin{theorem}[\Cref{thm L1 7}]\label{thmintro1} Let $\xi_0$ be the characteristic function of $\mathfrak{H}(\mathbf{Q}_p)\mathfrak{G}(\mathbf{Z}_p)$. The following are true:
 \begin{enumerate}
 \item\emph{(Integral abstract multiplicity one)} For $\xi\in C_c^\infty(\mathfrak{H}(\mathbf{Q}_p)\backslash \mathfrak{G}(\mathbf{Q}_p)/\mathfrak{G}(\mathbf{Z}_p),\mathbf{Z}[1/p])$ any integral vector of spherical level $\mathfrak{G}(\mathbf{Z}_p)$, we have 
 $$\mathfrak{z}(\xi)=\mathcal{P}_\xi^\mathrm{unv}\cdot \mathfrak{z}(\xi_0)\ \ \ \mathrm{with}\ \ \mathcal{P}_\xi^\mathrm{unv}\in \mathcal{R}_{\mathbf{Z}[1/p]}^{\mathfrak{G}_p}.$$
 \item\emph{(Optimality of integral abstract norm-relations)}
      For $\xi_1\in C_c^\infty(\mathfrak{H}(\mathbf{Q}_p)\backslash \mathfrak{G}(\mathbf{Q}_p)/\mathfrak{G}(\mathbf{Z}_p)[p],\mathbf{Z}[1/p])$ any integral vector of ramified level $\mathfrak{G}(\mathbf{Z}_p)[p]$, we have
     $$\mathrm{norm}^{\mathfrak{G}(\mathbf{Z}_p)[p]}_{\mathfrak{G}(\mathbf{Z}_p)}\mathfrak{z}(\xi_1)=\mathcal{P}_{\mathrm{Tr}(\xi_1)}^\mathrm{unv}\cdot\mathfrak{z}(\xi_0)$$
     with 
     $$\mathcal{P}_{\mathrm{Tr}(\xi_1)}^\mathrm{unv}\in\mathfrak{h}^{H_p}:=\begin{dcases*}
\left\langle p+1,\mathcal{T}_p\right\rangle\ ,\ &\text{ if }H \text{is non-split}\\
\langle p-1,{\mathcal{P}_p^{H}}'(p^{-1/2}) \rangle\ ,\ &\text{ if } H\text{ is split}
\end{dcases*}$$
where  $\mathfrak{h}^{H_p}$ is regarded as an ideal in $\mathcal{R}^{\mathfrak{G}_p}_{\mathbf{Z}[1/p]}$, $\mathcal{T}_p$ is the operator corresponding to $\left[\begin{smallmatrix}
    p & \\
    & 1
\end{smallmatrix}\right]$ and ${\mathcal{P}_p^{H}}'(X)$ is an explicit polynomial in $\mathcal{R}^{\mathfrak{G}_p}[X]$ that interpolates local $L$-factors of the unramified principal-series. Moreover for certain choices of $\xi$ both generators of $\mathfrak{h}^{H_p}$ are obtained in either case.
\end{enumerate}
 \end{theorem}
 We note that for trivial central character, the ideal $\mathfrak{h}^{H_p}$ coincides with $\langle[\mathfrak{G}(\mathbf{Z}_p):\mathfrak{G}(\mathbf{Z}_p)[p]],{\mathcal{P}_p^{H}}'(p^{-1/2})\rangle$ for both non-split and split $H$, in the appropriate quotient of the Hecke algebra. For more on this, we refer to \Cref{rem non-split ideal h}. Also see \Cref{lem euler factor} for the definition of ${\mathcal{P}_p^{H}}'$ in both split and non-split cases. The Hecke operators appearing in \Cref{thmintro1} are canonical, and universal with respect to the family of all unramified characters of $H(\mathbf{Q}_p)$ (\Cref{sec modules for a fixed character} and \Cref{sec univ prop}). For the definition of $\mathcal{P}_\xi^\mathrm{unv}$ and $\mathcal{P}_{\mathrm{Tr}(\xi_1)}^\mathrm{unv}$ see \Cref{def unv hecke ops}. For details on how the first part of the theorem produces an integral variant of Waldspurger's multiplicity one, we refer the reader to \Cref{sec int structures}. 

 To show that these Hecke operators do exist and are in fact canonical, we prove certain freeness results for the unramified modules in question, both at an integral and non-integral level (\Cref{sec 3}). These structure theorems allow us to construct a universal Shintani function for $(\mathfrak{G},\mathfrak{H})$ (\Cref{sec unv shintani}) whose integral theory is of separate interest. This universal Shintani function interpolates classical local Shintanini functions in the style of \cite{murase1996shintani}. Combining this with explicit Jacquet-Langlands \cite{jacquet2006automorphic} zeta integral constructions for the normalized toric periods (\Cref{seczeta}), we are able to realize these Hecke operators and study their integral properties in detail. The treatment for split and non-split tori is somewhat different as it will become apparent further on. We hope that \Cref{thmintro1} can have future applications to the integral theory of Euler system tame norm-relations for $(\GL_2,H)$, as well as motivate the study of this integral phenomenon in a more general setting.

Our second main result is concerned with the integral properties of unramified toric periods associated to modular forms. 

Fix once and for all a prime $\ell$ and an isomorphism $\iota:\mathbf{C}\simeq\overline{\mathbf{Q}}_\ell$. Let $E=\mathbf{Q}(\sqrt{-M})$ be a quadratic field and $\mathscr{H}$ be the globally non-split torus over $\mathbf{Q}$ given by $\mathrm{Res}_{E/\mathbf{Q}}\GL_1$ regarded as a subgroup $\GL_2$ (\Cref{sec global}). Finally, we let $\mathscr{G}:=\GL_2\times\mathscr{H}$ where we regard $\mathscr{H}$ as a subgroup of $\mathscr{G}$ under the usual embedding $h\mapsto (h,h^{-1}).$ 
For a normalized cuspidal eigenform $f\in S_k(\Gamma_1(N),\varepsilon)$ of even integral weight, we write $\pi_f$ for its associated automorphic representation as in \cite{gelbart2006automorphic}. let $S$ be a finite set of places containing $\{\infty,p|2\ell NM\}$ and let $\chi_f^S$ be a character of $\mathscr{H}(\mathbf{A}^S)$ taking values in a number field, and whose central character agrees with $\varepsilon^S$ (the $S$-finite part of the adelization of $\varepsilon$). By a famous theorem of Waldspurger \cite{Waldspurger}, there is a unique normalized period 
$$\mathcal{Z}_f\in\mathrm{Hom}_{\mathscr{H}(\mathbf{A}^S)}(\pi_f^S\boxtimes\chi_f^S,\mathbf{1}).$$

We are interested in its integral behavior when evaluated on suitable lattices.
Let $L_f$ be the composite of the number field of $f$ and $\chi_f^S$. Let $\mathbf{L}_f/\mathbf{Q}_\ell$ be the smallest $\ell$-adic number field containing the image of $L_f$ under $\iota.$ For each $p\notin S$, $\mathscr{H}$ is an unramified maximal torus of $\GL_2$ over $\mathbf{Q}_p$, and we once again consider the ramified level subgroup $\mathscr{H}(\mathbf{Z}_p)[p]$ (\Cref{def level subgroup}). For a finite set of primes $S'$ disjoint from $S$, we write $\mathscr{G}(\hat{\mathbf{\mathbf{Z}}}^S)[S']:=\{\prod_{p\in S'}\mathscr{G}(\mathbf{Z}_p)[p]\}\times\{\prod_{p\not\in S'}\mathscr{G}(\mathbf{Z}_p)\}$ where as in \Cref{thmintro1}, each $\mathscr{G}(\mathbf{Z}_p)[p]:=\GL_2(\mathbf{Z}_p)\times \mathscr{H}(\mathbf{Z}_p)[p]$. Finally, we write $\mathrm{vol}_{\mathscr{H}(\mathbf{A}^S)}:=\prod_{p\notin S}\mathrm{vol}_{\mathscr{H}(\mathbf{Q}_p)}$ where each $\mathrm{vol}_{\mathscr{H}(\mathbf{Q}_p)}$ is the normalized Haar measure on $\mathscr{H}(\mathbf{Q}_p)$ giving $\mathscr{H}(\mathbf{Z}_p)$ volume $1$.
\begin{theorem}[Integrality of toric periods for modular forms; \Cref{thm integr of toric periods}]\label{thmintro2}
Let $f\in S_k(\Gamma_1(N),\varepsilon)$ be a normalized cuspidal eigenform of even integral weight. Then the following are true:
\begin{enumerate}
    \item For any $g\in\mathscr{G}(\mathbf{A}^S)$ and any $C_g\in\mathrm{vol}_{\mathscr{H}(\mathbf{A}^S)}(\mathscr{H}(\mathbf{A}^S)\cap g \mathscr{G}(\hat{\mathbf{Z}}^S)g^{-1})\cdot\mathcal{O}_{\mathbf{L}_f}$, the linear form $\mathcal{Z}_f$ satisfies 
    $$C_g\mathcal{Z}_f(g\cdot W_{\pi_f^S\boxtimes\chi_f^S}^\mathrm{sph})\in\mathcal{O}_{\mathbf{L}_f}.$$
    \item If moreover $C_g\in \mathrm{vol}_{\mathscr{H}(\mathbf{A}^S)}(\mathscr{H}(\mathbf{A}^S)\cap g \mathscr{G}(\hat{\mathbf{Z}}^S)[S']g^{-1})\cdot\mathcal{O}_{\mathbf{L}_f}$ for a finite set of primes $S'$ disjoint from $S$, then the linear form $\mathcal{Z}_f$ satisfies
    $$\scalemath{0.95}{C_g\mathcal{Z}_f(g\cdot W_{\pi_f^S\boxtimes\chi_f^S}^\mathrm{sph})\in\left\{\prod_{\substack{p\in S'\\ \mathrm{inert\ in\ }E}}\left\langle p+1, a_p(f)\right\rangle\right\}\times\left\{ \prod_{\substack{p\in S'\\ \mathrm{split\ in\ }E}}\left\langle p-1, L_p^{-1}\left(f;\chi_{f,p}(\left[\begin{smallmatrix}
        1/p & \\
        & 1
    \end{smallmatrix}\right])\right) \right\rangle\right\}\subseteq \mathcal{O}_{\mathbf{L}_f}}$$
    where $a_p(f)$ is the $p$-th Fourier coefficient of $f$, and $\scalemath{0.93}{L_p^{-1}(f;X):=1-a_p(f)p^{-\tfrac{k}{2}}X+\varepsilon(p)p^{-1}X^2\in\mathcal{O}_{\mathbf{L}_f}[X].}$
\end{enumerate}
\end{theorem}
 We remark that the ideals appearing in the second part of the theorem at each prime $p\in S'$, are precisely the ideals given by the image of the $\mathfrak{h}^{\mathscr{H}(\mathbf{Q}_p)}$ of \Cref{thmintro1}, under the spherical Hecke eigensystem $\mathcal{R}^{\mathscr{H}(\mathbf{Q}_p)}_{\mathbf{Z}[1/p]}\rightarrow(\pi_{f,p}\boxtimes \chi_{f,p})^{\mathscr{G}(\mathbf{Z}_p)}.$ 
\subsection{Acknowledgements} I am extremely grateful to my PhD supervisor David Loeffler for introducing me to this problem, for his continuous guidance and support, and for all his valuable comments and suggestions regarding earlier versions of this paper. I would also like to thank Ju-Feng Wu for all his insightful comments and ongoing support. Finally, I would like to thank Robert Kurinczuk for useful discussions. 
{
  \hypersetup{linkcolor=black}
  \tableofcontents
}
\section{Preliminaries}
\subsection{Group schemes, Hecke algebras and Hecke modules}\label{GHM}
Let $p$ be an odd prime and $H$ and unramified maximal torus of $G:=\GL_{2}$ over $\mathbf{Q}_p$. Throughout the paper, we will simply refer to split and non-split tori to mean split and non-split over $\mathbf{Q}_p$.
Let $E=\mathbf{Q}_p(\sqrt{-D})$ be the unramified quadratic extension of $\mathbf{Q}_p$ where we take $D$ to be an element of $\mathbf{Z}_p^\times$. The torus $H$ is given (up to conjugation) by
$$H=\begin{cases}
    \mathrm{Res}_{\mathcal{O}_E/\mathbf{Z}_p}\GL_1\ ,\ H \text{ non-split}\\
    \GL_1\times \GL_1 \ ,\ H \text{ split}
\end{cases}$$
with the following choices of embeddings
$$\mathrm{Res}_{\mathcal{O}_E/\mathbf{Z}_p}\GL_1\hookrightarrow G\ ,\ a+b\sqrt{-D}\mapsto\left[\begin{smallmatrix}
    a & b\\
    -bD & a
\end{smallmatrix}\right]$$
$$\GL_1\times \GL_1 \hookrightarrow G\ ,\ (a,b)\mapsto \left[\begin{smallmatrix}
    a & \\
    & b
\end{smallmatrix}\right].$$
\noindent We define the groups $$G_p:=G(\mathbf{Q}_p),\ G_p^\circ:=G(\mathbf{Z}_p),\ H_p:=H(\mathbf{Q}_p),\ H_p^\circ:=H(\mathbf{Z}_p).$$

\noindent The subgroups $G_p^\circ $ and $H_p^\circ$ are maximal open compacts in $G_p$ and $H_p$ respectively. We fix $\mathbf{Q}$-valued Haar measures on $G_p$ and $H_p$ such that the volumes of $G_p^\circ$ and $H_p^\circ$ are one, respectively. We denote by $\mathcal{H}(G_p)=\mathcal{H}(G_p,\mathbf{C})$ the full Hecke algebra of $G_p$, consisting of locally constant, compactly supported $\mathbf{C}$-valued functions with coefficients. We regard this as an algebra through the usual convolution 
$$(\phi_1 \cdot_G \phi_2)(x):=\int_{G_p}\phi_1(g)\phi_2(g^{-1}x)dg.$$
Any smooth $G_p$-representation can be regarded as a $\mathcal{H}(G_p)$-module via the usual formula, $$\phi\cdot v :=\int_{G_p}\phi(g) g\cdot v dg\ ,\ \phi\in\mathcal{H}(G_p), v\in V .$$This is of course compatible with the $\cdot_G$ convolution of $\mathcal{H}(G_p)$.
For an element $\phi\in\mathcal{H}(G_p)$, we denote by $\phi^{'}$, the composition of $\phi$ with the involution of $G_p$, given by inversion. That is, $\phi^{'}(g):=\phi(g^{-1})$ for all $g\in G_p$.\\

\noindent We write $\mathcal{H}_{G_p}^\circ=\mathcal{H}(G_p^\circ\backslash G_p/G_p^\circ,\mathbf{C})$ for the spherical Hecke algebra of $G_p$, of locally constant, compactly supported, $\mathbf{C}$-valued, $G_p^\circ$-bi-invariant functions on $G_p$. It is well known, that $\mathcal{H}_{G_p}^\circ$ is a commutative algebra of Weyl-group invariant polynomials in the Satake parameters, and it is explicitly given by $\mathbf{C}[\mathcal{T}_p,\mathcal{S}_p^{\pm 1}]$, where $$\mathcal{T}_p:=\ch(G_p^\circ\left[\begin{smallmatrix}
    p & \\
     & 1
\end{smallmatrix}\right]G_p^\circ)\ \ ,\ \ \mathcal{S}_p:=\ch(\left[\begin{smallmatrix}
    p & \\
     & p
\end{smallmatrix}\right]G_p^\circ).$$
The spherical Hecke algebra of $H_p$ which we denote by $\mathcal{H}_{H_p}^\circ$ is explicitly given by $\mathbf{C}[X_p^{\pm 1}]$ where $$X:=\ch(\left[\begin{smallmatrix}
    p & \\
     & p
\end{smallmatrix}\right] H_p^\circ)$$
if $H$ is non-split. On the other hand, if $H$ is split, then $\mathcal{H}_{H_p}^\circ$ is given by $\mathbf{C}[A_p^{\pm 1}, B_p^{\pm 1}]$, where $$A_p:=\ch(\left[\begin{smallmatrix}
    p& \\
    & 1
\end{smallmatrix}\right] H_p^\circ)\ \ ,\ \ B_p:=\ch(\left[\begin{smallmatrix}
    1 & \\
     & p
\end{smallmatrix}\right] H_p^\circ).$$
If $R$ is a $\mathbf{Z}$-subalgebra of $\mathbf{C}$ we write $\mathcal{H}_{G_p}^\circ(R)$ for the $R$-algebra consisting of elements of $\mathcal{H}_{G_p}^\circ$ that take values in $R$. The algebra $\mathcal{H}_{G_p}^\circ(R)$ is still generated as an $R$-algebra by $\mathcal{S}_p^{\pm 1}$ and $\mathcal{T}_p$. In other words, $\mathcal{H}_{G_p}^\circ(R)=R[\mathcal{S}_p^{\pm 1},\mathcal{T}_p].$ Furthermore, for an $R$-module $M$ contained in $\mathbf{C}$, we write $\mathcal{H}_{G_p}^\circ(M)$ for the $\mathcal{H}_{G_p}^\circ(R)$-module given by elements of $\mathcal{H}_{G_p}^\circ$ that take values in $M$.\\
\\
Throughout, we will denote the center of $G_p$ by $Z_p$. We also write $B_p$ for the upper triangular Borel of $G_p$, $N_p$ for its unipotent radical, and $A_p$ for the subgroup $\left[\begin{smallmatrix}
    \mathbf{Q}_p^\times & \\
     & 1
\end{smallmatrix}\right]$. This way, we have the well-known Iwasawa decomposition $G_p=B_pG_p^\circ=Z_pA_pN_pG_p^\circ$. For ease of notation, when it is clear what we mean, we will simply denote an element $\left[\begin{smallmatrix}
    z & \\
    & z
\end{smallmatrix}\right]$ of the center of $G_p$ by $z$. Furthermore, we denote by $H_Z$, the quotient of $H$ by the center of $G$. Thus $$\#H_Z(\mathbf{F}_p)=\begin{cases}
    p+1\ ,\ \ H \textit{ non-split}\\
    p-1\ ,\ \ H \textit{ split}
\end{cases}.$$

\noindent We define the group scheme $\mathfrak{G}:= G\times H$, and the subgroup scheme $\mathfrak{H}$ to be the image of $H$ in $\mathfrak{G}$ under the embedding $h\mapsto (h,h^{-1})$. The groups of points $\mathfrak{G}_p,\mathfrak{G}_p^\circ,\mathfrak{H}_p$ and $\mathfrak{H}_p^\circ$ are then defined in the same way as before. The spherical Hecke algebra of $\mathfrak{G}_p$, which we denote by $\mathcal{H}_{\mathfrak{G}_p}^\circ=\mathcal{H}(\mathfrak{G}_p^\circ\backslash \mathfrak{G}_p/\mathfrak{G}_p^\circ,\mathbf{C})$, is canonically identified with $\mathcal{H}_{G_p}^\circ \otimes \mathcal{H}_{H_p}^\circ$ and is given by $\mathbf{C}[\mathcal{T}_p,\mathcal{S}_p^{\pm1 },X_p^{\pm 1}]$ if $H$ is non-split, and by $\mathbf{C}[\mathcal{T}_p,\mathcal{S}_p^{\pm 1}, A_p^{\pm 1}, B_p^{\pm 1}]$ if $H$ is split. The involution $(-)^{'}$ defined on $\mathcal{H}_{G_p}^\circ$ extends to an involution of $\mathcal{H}_{\mathfrak{G}_p}^\circ$ in the obvious way, by $(\phi\otimes f)^{'}=\phi^{'}\otimes f^{'}$. 
\subsubsection{Hecke modules}\label{UHM}
We define the following three spaces of functions on $\mathfrak{G}_p$ that play a central role in this work.
\begin{enumerate}
\item The space $C_{\mathfrak{H}_p\backslash c}^{\infty,\mathrm{sph}}:=C_{c}^\infty(\mathfrak{H}_p\backslash \mathfrak{G}_p/\mathfrak{G}_p^\circ,\mathbf{C})$, will denote the space of $\mathbf{C}$-valued functions on $\mathfrak{G}_p$ that are compactly supported modulo $\mathfrak{H}_p$ on the left, right $\mathfrak{G}_p^\circ$-invariant and left $\mathfrak{H}_p$-invariant.
\item The lattice $\mathcal{L}^{H_p}:=C_{c}^\infty(\mathfrak{H}_p\backslash \mathfrak{G}_p/\mathfrak{G}_p^\circ,\mathbf{Z}[1/p])$ will denote functions in $C_{\mathfrak{H}_p\backslash c}^{\infty,\mathrm{sph}}$ that are valued in $\mathbf{Z}[1/p]$.
\item The lattice $\mathcal{L}_1^{H_p}\subseteq\mathcal{L}^{H_p}$ will denote functions in $C_{\mathfrak{H}_p\backslash c}^{\infty,\mathrm{sph}}$ that are valued in $\#H_Z(\mathbf{F}_p)\mathbf{Z}[1/p]$ on
$H_p\times\{1\}$, and in $\mathbf{Z}[1/p]$ everywhere else.
\end{enumerate}
\noindent We write 
$$\mathcal{R}^{\mathfrak{G}_p}:=\begin{cases}
    \mathcal{H}_{\mathfrak{G}_p}^\circ/(\mathcal{S}_p-X_p)\ ,\ H\text{ non-split}\\
    \mathcal{H}_{\mathfrak{G}_p}^\circ/(\mathcal{S}_p-A_pB_p)\ ,\ H \text{ split}
\end{cases}\ \ \ \ \mathcal{R}^{\mathfrak{G}_p}_{\mathbf{Z}[1/p]}:=\begin{cases}
    \mathcal{H}_{\mathfrak{G}_p}^\circ(\mathbf{Z}[1/p])/(\mathcal{S}_p-X_p)\ ,\ H\text{ non-split}\\
    \mathcal{H}_{\mathfrak{G}_p}^\circ(\mathbf{Z}[1/p])/(\mathcal{S}_p-A_pB_p)\ ,\ H \text{ split}
\end{cases}.$$
We will often simply identify $\mathcal{R}^{\mathfrak{G}_p}$ with $\mathbf{C}[\mathcal{T}_p,\mathcal{S}_p^{\pm 1}]$ and $\mathcal{R}_{\mathbf{Z}[1/p]}^{\mathfrak{G}_p}$ with $\mathbf{Z}[1/p][\mathcal{T}_p,\mathcal{S}_p^{\pm 1}]$ if $H$ is non split. Similarly, we will often identify $\mathcal{R}^{\mathfrak{G}_p}$ with $\mathbf{C}[\mathcal{T}_p,\mathcal{S}_p^{\pm 1}, A_p^{\pm 1}]$ and $\mathcal{R}_{\mathbf{Z}[1/p]}^{\mathfrak{G}_p}$ with $\mathbf{Z}[1/p][\mathcal{T}_p,\mathcal{S}_p^{\pm 1}, A_p^{\pm 1}]$ if $H$ is split. It is worth noting that $\mathcal{R}^{\mathfrak{G}_p}$ can also be regarded as the spherical Hecke algebra of the group of $\mathbf{Q}_p$-points of the group scheme $\mathfrak{G}/\{(z,z^{-1})\ |\ z\in Z_G\}$. Most of the representation theory carried out in this paper factors through this quotient, since we will mostly be interested in $\mathfrak{G}$-representations that have non-trivial $\mathfrak{H}$-invariant periods.\\
\\
\noindent The space $C_{\mathfrak{H}_p\backslash c}^{\infty,\mathrm{sph}}$ is naturally a module over the Hecke algebra $\mathcal{H}_{\mathfrak{G}_p}^\circ$, with action given through the convolution
$$\left((\phi\otimes f)*\xi\right)(x,y):=\int_{G_p}\int_{H_p} \phi(g)f(h)\xi(xg,yh)\ dh\ dg\ ,\ \ \phi\otimes f\in\mathcal{H}_{\mathfrak{G}_p}^\circ, \xi\in C_{\mathfrak{H}_p\backslash c}^{\infty,\mathrm{sph}}.$$
It is clear that for $H$ non-split (resp. split), the module $C_{\mathfrak{H}_p\backslash c}^{\infty,\mathrm{sph}}$ is annihilated by $\mathcal{S}_p-X_p$ (resp. $\mathcal{S}_p-A_pB_p$). Thus $C_{\mathfrak{H}_p\backslash c}^{\infty,\mathrm{sph}}$ is naturally a module over $\mathcal{R}^{\mathfrak{G}_p}$. Similarly, the lattice $\mathcal{L}^{H_p}$ is naturally a module over $\mathcal{R}_{\mathbf{Z}[1/p]}^{\mathfrak{G}_p}$.\\
We denote by $\mathbf{C}[\mathfrak{G}_p/\mathfrak{G}_p^\circ]$, the space of functions in the full Hecke algebra of $\mathfrak{G}_p$ that are right $\mathfrak{G}_p^\circ$-invariant. This admits an action of $\mathfrak{H}_p$ through $(h\cdot \xi)(x,y):=\xi(h^{-1}\cdot (x,y))=\xi(h^{-1}x,hy)$. We write $\mathbf{C}[\mathfrak{G}_p/\mathfrak{G}_p^\circ]_{\mathfrak{H}_p}$ for its $\mathfrak{H}_p$-coinvariants and denote the class of an element by $[-]$. This is naturally a module over $\mathcal{R}^{\mathfrak{G}_p}$ through the $*$-action defined above. It is clear from the definitions that for $\xi\in\mathbf{C}[\mathfrak{G}_p/\mathfrak{G}_p^\circ]$ and $\phi\otimes f \in\mathcal{H}_{\mathfrak{G}_p}^\circ$, we have $(\phi\otimes f)* \xi=\xi \cdot (\phi\otimes f)^{'}$ where the $\cdot$ action denotes the standard convolution in $\mathcal{H}(\mathfrak{G}_p)=\mathcal{H}(G_p)\otimes \mathcal{H}(H_p)$. This fact will be used throughout the paper without any further reference to it. 

\begin{defn}\label{maps}
   Let $V$ be a smooth $\mathbf{C}$-linear $\mathfrak{G}_p$-representation. A $\mathbf{C}$-linear map $$\zeta:\mathcal{H}(\mathfrak{G}_p)\longrightarrow V$$
   is $(\mathfrak{H}_p\times \mathfrak{G}_p)$-equivariant if it is equivariant with respect to the following actions:
   \begin{enumerate}
       \item $\mathfrak{G}_p$ acts on the left via $g\cdot \xi:=\xi((-)g)$ and on the right by its assigned action on $V$.
       \item $\mathfrak{H}_p$ acts on the left via $h\cdot \xi:=\xi(h^{-1}(-))$ and trivially on the right.
   \end{enumerate}
\end{defn}
\noindent Clearly such a map factors through the $\mathfrak{H}_p$-coinvariants of $\mathcal{H}(\mathfrak{G}_p)$. The class of $(\mathfrak{H}_p\times \mathfrak{G}_p)$-equivariant maps can be identified with $\Hom_{\mathfrak{H}_p}(V^\vee, \mathbf{1})$ (c.f \cite{loeffler2021euler} Section $3.9$) where $V^\vee$ denotes the smooth dual of $V$ regarded as a $\mathfrak{G}_p$-representation via $g\cdot f:=f(g^{-1}(-))$. Later in the paper, we will actually need to construct an explicit element of $\Hom_{\mathfrak{H}_p}(V^\vee,\mathbf{1})$ in the case $V=\pi_p\boxtimes \chi_p$ where $\pi_p$ is an unramified principal series representation of $G_p$ of Whittaker type, and $\chi_p$ is an unramified character of $H_p$ with $\omega_{\pi_p}=\chi_p|_{\mathbf{Q}_p^\times}$. Being able to explicitly construct and study elements of this space for a big enough family of representations will be important in later sections. 

Given a smooth $\mathfrak{G}_p$-representation $V$, and open compact subgroups $U_1\subseteq U_2\subseteq\mathfrak{G}_p^\circ$, we write $$\mathrm{norm}^{U_1}_{U_2}: V^{U_2}\rightarrow V^{U_1},\ v\mapsto\sum_{k\in U_1/U_2}k\cdot v.$$
 \begin{prop}\label{equiv}
        For any $[\xi]\in\mathbf{C}[\mathfrak{G}_p/\mathfrak{G}_p^\circ]_{\mathfrak{H}_p}$, and $\zeta$ a $(\mathfrak{H}_p\times \mathfrak{G}_p)$-equivariant map into a smooth $\mathfrak{G}_p$-representation $V$, we have 
        $$\zeta(\mathcal{P} * [\xi])=\ \mathcal{P}\cdot \zeta([\xi])$$
        for all $\mathcal{P}\in\mathcal{H}_{\mathfrak{G}_p}^\circ$.
        \begin{proof}
            This follows from a standard unraveling of definitions.
        \end{proof}
    \end{prop}
\subsubsection{From invariants to coinvariants}\label{sec inv to coinv}
    For any $(V,\zeta)$ as in \Cref{maps}, there is a natural linear map
    $\mathfrak{z}:C_c^\infty(\mathfrak{H}_p\backslash \mathfrak{G}_p)\longrightarrow V$ given by the composition of the maps
    $$C_c^\infty(\mathfrak{H}_p\backslash \mathfrak{G}_p)\simeq_i \mathcal{H}(\mathfrak{G}_p)_{\mathfrak{H}_p}\overset{\zeta}{\longrightarrow}V.$$
    The isomorphism $i$ is given by sending $\ch(\mathfrak{H}_p gU)$ to $\vol_{\mathfrak{H}_p}(\mathfrak{H}_p\cap gUg^{-1})^{-1}\cdot[\ch(gU)]$ for any open compact subgroup $U\subseteq \mathfrak{G}_p^\circ$. It has an inverse given by integration over $\mathfrak{H}_p$; i.e. $[\xi]\mapsto(g\mapsto\int_{\mathfrak{H}_p}\xi(hg)\ dh).$ It is clear that we have an induced map at each finite level $U\subseteq\mathfrak{G}(\mathbf{Q}_p)$ (open compact subgroup)
    $$C_c^\infty(\mathfrak{H}_p\backslash \mathfrak{G}_p/U)\simeq_i \mathbf{C}[\mathfrak{G}_p/U]_{\mathfrak{H}_p}\overset{\zeta}{\longrightarrow}V^U.$$
    \subsubsection{Double coset decompositions}\label{secDCD}
\begin{lem}\label{Dns}
    If $H$ is non-split then we have the following double coset decomposition
    $$\mathfrak{G}_p=\bigcup_{\substack{\lambda\in\mathbf{Z}_{\geq 0}\\ a\in\mathbf{Z}}}\mathfrak{H}_p\left(s(\lambda),p^a\right)\mathfrak{G}_p^\circ$$
    where $s(\lambda):=\left[\begin{smallmatrix}
        p^\lambda & \\
         & 1
    \end{smallmatrix}\right]$. Furthermore, the decomposition is disjoint.
    \begin{proof}
        \Cref{lemDns} in appendix.
    \end{proof}
\end{lem}
\begin{lem}\label{Ds}
    If $H$ is split then we have the following double coset decomposition
$$\mathfrak{G}_p=\bigcup_{\substack{\lambda\in\mathbf{Z}_{\geq 0}\\ a,b\in\mathbf{Z}}}\mathfrak{H}_p\left(n_0s(\lambda),\left[\begin{smallmatrix}
        p^a & \\
        & 1
    \end{smallmatrix}\right] p^b\right)\mathfrak{G}_p^\circ$$
    where $n_0:=\left[\begin{smallmatrix}
        1 & 1\\
        & 1
    \end{smallmatrix}\right]$ and $s(\lambda)$ is as in \Cref{Dns}. Furthermore, the decomposition is disjoint.
    \begin{proof}
        \Cref{lemDs} in appendix.
    \end{proof}
\end{lem}

\begin{defn}\label{char}
\begin{enumerate}
 \item If $H$ is non-split, $\gamma_{\lambda,a}:=(s(\lambda),p^a)\in \mathfrak{G}_p$, and $\xi_{\lambda,a}$ is the characteristic function attached to the double coset $\mathfrak{H}_p\gamma_{\lambda,a}\mathfrak{G}_p^\circ.$
    for $\lambda\in\mathbf{Z}_{\geq 0}$ and $a\in\mathbf{Z}$. 
    \item If $H$ is split, then $\gamma_{\lambda,a,b}:=(n_0s(\lambda),\left[\begin{smallmatrix}
        p^a & \\
         & 1
    \end{smallmatrix}\right]p^b)\in\mathfrak{G}_p$ and $\xi_{\lambda,a,b}$ is the characteristic function attached to the double coset $\mathfrak{H}_p\gamma_{\lambda,a,b}\mathfrak{G}_p^\circ$
    for $\lambda\in\mathbf{Z}_{\geq 0}$ and $a,b\in\mathbf{Z}$.
    \end{enumerate}
\end{defn}
 \noindent \Cref{Dns} and \Cref{Ds} give us the following decompositions into subspaces 
$$C_{\mathfrak{H}_p\backslash c}^{\infty,\mathrm{sph}}=\begin{cases}
\bigoplus_{\substack{\lambda\in\mathbf{Z}_{\geq 0} \\ a\in\mathbf{Z}}}\mathbf{C}\cdot \xi_{\lambda,a}\ ,\ H \text{ non-split}\\
    \bigoplus_{\substack{\lambda\in\mathbf{Z}_{\geq 0}\\ a,b\in\mathbf{Z}}}\mathbf{C}\cdot \xi_{\lambda,a,b}\ ,\ H\text{ split}
\end{cases}\ \ ,\ \ \mathcal{L}^{H_p}=\begin{cases}
\bigoplus_{\substack{\lambda\in\mathbf{Z}_{\geq 0} \\ a\in\mathbf{Z}}}\mathbf{Z}[1/p]\cdot \xi_{\lambda,a}\ ,\ H \text{ non-split}\\
    \bigoplus_{\substack{\lambda\in\mathbf{Z}_{\geq 0}\\ a,b\in\mathbf{Z}}}\mathbf{Z}[1/p]\cdot \xi_{\lambda,a,b}\ ,\ H\text{ split}
\end{cases}$$
$$\mathcal{L}_1^{H_p}=\begin{cases}\ \ \left(\bigoplus
_{a\in\mathbf{Z}}\#H_Z(\mathbf{F}_p)\mathbf{Z}[1/p]\cdot \xi_{0,a}\right) &\oplus\  \left(\bigoplus_{\substack{\lambda\in\mathbf{Z}_{\geq 1}\\ a\in\mathbf{Z}}}\mathbf{Z}[1/p]\cdot \xi_{\lambda,a}\right)\ ,\ H\text{ non-split}\\
\left(\bigoplus
_{a,b\in\mathbf{Z}}\#H_Z(\mathbf{F}_p)\mathbf{Z}[1/p]\cdot \xi_{0,a,b}\right)\ &\oplus\ \left(\bigoplus_{\substack{\lambda\in\mathbf{Z}_{\geq 1}\\ a,b\in\mathbf{Z}}}\mathbf{Z}[1/p]\cdot \xi_{\lambda,a,b}\right)\ ,\ H\text{ split}

\end{cases}$$

\section{Universal Hecke modules}\label{sec 3}
\subsection{Structure theorems}\label{Sec the modules C and L} We note that the results of this section hold with $\mathbf{C}$ replaced by any algebraically closed field of characteristic zero.
\begin{thm}\label{thmCunivfree}
    The modules $C_{\mathfrak{H}_p\backslash c}^{\infty,\mathrm{sph}}$ and $\mathcal{L}^{H_p}$ are free of rank one over $\mathcal{R}^{\mathfrak{G}_p}$ and $\mathcal{R}_{\mathbf{Z}[1/p]}^{\mathfrak{G}_p}$ respectively.
    \end{thm}

    Before going into the proof, we note that for the split case, the result at a non-integral level (i.e. only for the module $C_{\mathfrak{H}_p\backslash c}^{\infty,\mathrm{sph}}$) can be deduced from a very general result of Sakellarides \cite[Corollary $8.0.4$] {sakellaridis2013spherical}. Here we give a direct proof that covers both non-split and split cases, and more importantly gives us a handle on integrality. Our approach is inspired by the proof of uniqueness of unramified Shintani functions in the style of \cite{2003}.
    \begin{proof}
    We firstly prove the result in the case where $H$ is non-split. Note that in this case $\ch(\mathfrak{H}_p\mathfrak{G}_p^\circ)=\xi_{0,0}$. It is clear by construction that 
    $$\mathcal{S}_p^{-a}*\xi_{0,0}=\xi_{0,a}.$$
    We write $\phi_{s(\lambda)}$ for the characteristic function $\ch(G_p^\circ s(\lambda)^{-1} G_p^\circ)\in\mathcal{H}_{G_p}^\circ$. We then claim that
    \begin{align}\phi_{s(\lambda)^{-1}}*\xi_{0,0}=\xi_{\lambda,0}.
    \end{align}
    The function 
    $\left(\phi_{s(\lambda)^{-1}}*\xi_{0,0}\right)(x,y)=\int_{G_p^\circ s(\lambda)^{-1} G_p^\circ} \xi_{0,0}(xg,y)dg$ 
    is supported on $\mathfrak{H}_p\mathfrak{G}_p^\circ(s(\lambda),1)\mathfrak{G}_p^\circ$. We decompose 
    $$G_p^\circ s(\lambda) G_p^\circ=\left(\bigsqcup_{\beta\in \mathbf{Z}/p^{\lambda}\mathbf{Z}}\left[\begin{smallmatrix}
            p^{\lambda} & \beta\\
            & 1
        \end{smallmatrix}\right] G_p^\circ\right)  \sqcup \left(\bigsqcup_{\substack{0<i<\lambda\\ \\ \beta\in(\mathbf{Z}/p^{i}\mathbf{Z})^\times}}\left[\begin{smallmatrix}
            p^{i} & \beta\\
            &p^{\lambda-i}
        \end{smallmatrix}\right] G_p^\circ\right)\sqcup\left[\begin{smallmatrix}
            1 & \\
            &p^{\lambda}
        \end{smallmatrix}\right] G_p^\circ.$$
        We now use the Iwasawa decomposition algorithm to show that all these left cosets get absorbed into one single double coset, namely $H_p^\circ s(\lambda) G_p^\circ$, upon multiplying on the left with $H_p^\circ$.
        Firstly, one can check that 
        \begin{align*}\left[\begin{smallmatrix}
            1 & b\\
            -bD & 1
        \end{smallmatrix}\right] s(\lambda)&\equiv\left[\begin{smallmatrix}
            p^\lambda & b\\
             & 1
        \end{smallmatrix}\right]\left[\begin{smallmatrix}
            1+b^2 D & \\
             & 1
        \end{smallmatrix}\right]\mod G_p^\circ\\
        &\equiv \left[\begin{smallmatrix}
            p^\lambda & b\\
             & 1
        \end{smallmatrix}\right] \mod G_p^\circ
        \end{align*}
        for all $b\in\mathbf{Z}_p$. As we range $b\in\mathbf{Z}/p^\lambda \mathbf{Z}$, this covers all right cosets appearing in the first disjoint union above. 
        Secondly, one can also check that for $0<i<\lambda$ and $b\in\mathbf{Z}_p^\times$, we have
        \begin{align*}
            \left[\begin{smallmatrix}
                p^i & b\\
                -bD & p^i
            \end{smallmatrix}\right]s(\lambda)&\equiv \left[\begin{smallmatrix}
                p^{\lambda-i} & b\\
                & p^i
            \end{smallmatrix}\right]\left[\begin{smallmatrix}
                p^{2i}+b^2D & \\
                & 1
            \end{smallmatrix}\right]\mod G_p^\circ\\
            &\equiv\left[\begin{smallmatrix}
                p^{\lambda-i} & b\\
                & p^i
            \end{smallmatrix}\right]\mod G_p^\circ.
        \end{align*}
        This covers all left cosets in the second disjoint union above. Finally we have $$\left[\begin{smallmatrix}
             & 1\\
             -D& 
        \end{smallmatrix}\right]s(\lambda)\equiv \left[\begin{smallmatrix}
            1 & \\
            & p^\lambda 
        \end{smallmatrix}\right]\mod G_p^\circ.$$
        This precisely implies that 
        $\mathfrak{H}_p\mathfrak{G}_p^\circ(s(\lambda),1)\mathfrak{G}_p^\circ=\mathfrak{H}_p(s(\lambda),1)\mathfrak{G}_p^\circ.$ To prove $(1)$ it remains to show that $$\int_{G_p^\circ s(\lambda)^{-1} G_p^\circ} \xi_{0,0}(s(\lambda)g,1)\ dg$$is equal to $1$. It's not hard to see that the value of this integral  is given explicitly by $$\#\left\{ h\in\mathfrak{H}_p/\mathfrak{H}_p^\circ\ |\ (s(\lambda)^{-1},1)h\in \mathfrak{G}_p^\circ (s(\lambda)^{-1},1)\mathfrak{G}_p^\circ\right\}=1.$$
        This proves $(1)$. Using \Cref{Dns}, it is now easy to deduce that $C_{\mathfrak{H}_p\backslash c}^{\infty,\mathrm{sph}}$ and $\mathcal{L}^{H_p}$ are free of rank one over $\mathcal{R}^{\mathfrak{G}_p}$ and $\mathcal{R}_{\mathbf{Z}[1/p]}^{\mathfrak{G}_p}$ respectively, generated by $\xi_{0,0}$.\\
        \\
        \noindent We now prove the result in the case where $H$ is split. In this case $\ch(\mathfrak{H}_p\mathfrak{G}_p^\circ)=\xi_{0,0,0}$. It once again follows straight from the definitions that 
        $$\mathcal{S}_p^{-n}A_p^{-m} * \xi_{\lambda,a,b}=\xi_{\lambda,a+m,b+n}$$
        for all $\lambda\in\mathbf{Z}_{\geq 0}, a,b,n,m\in\mathbf{Z}$. We once again want to compute $\phi_{s(\lambda)^{-1}}* \xi_{0,0,0}$, which is again supported on $\mathfrak{H}_p\mathfrak{G}_p^\circ(s(\lambda),1)\mathfrak{G}_p^\circ$. Things in the split case however are more complicated since $v_p\circ \det$ does not completely detect if an element of $H_p$ lies in $H_p^\circ$ or not. In other words, $\mathfrak{H}_p\mathfrak{G}_p^\circ(s(\lambda),1)\mathfrak{G}_p^\circ$ strictly contains $\mathfrak{H}_p(n_0s(\lambda),1)\mathfrak{G}_p^\circ$. We have a filtration 
        $$\mathcal{A}_0\subseteq \mathcal{A}_1\subseteq \dots \subseteq \mathcal{A}_n\subseteq \dots \subset C_{\mathfrak{H}_p\backslash c}^{\infty,\mathrm{sph}}$$
        where $\mathcal{A}_n=\bigoplus_{\substack{0\leq\lambda\leq n \\ a,b\in\mathbf{Z}}}\mathbf{C}\cdot \xi_{\lambda,a,b}$. Our goal is to relate the support $\mathfrak{H}_p\mathfrak{G}_p^\circ (s(\lambda),1)\mathfrak{G}_p^\circ$ to the above filtration in a way that allows us to proceed inductively. Considering the double coset decomposition of $G_p^\circ s(\lambda) G_p^\circ$ given above, we see that the double coset $H_p G_p^\circ$ absorbs precisely the left cosets $\left[\begin{smallmatrix}
            p^\lambda & \\
             & 1
        \end{smallmatrix}\right]G_p^\circ$ and $\left[\begin{smallmatrix}
            1 & \\
             & p^\lambda
        \end{smallmatrix}\right]G_p^\circ$. For fixed $0<i\leq \lambda$, the double coset $H_p\left[\begin{smallmatrix}
            p^i & 1\\
            & p^{\lambda-i}
        \end{smallmatrix}\right]G_p^\circ=H_pn_0 s(i) G_p^\circ$ absorbs precisely the left cosets $\left[\begin{smallmatrix}
            p^i & (\mathbf{Z}/p^i\mathbf{Z})^\times \\
            & p^{\lambda-i}
        \end{smallmatrix}\right]G_p^\circ$ and $\left[\begin{smallmatrix}
            p^\lambda & p^{\lambda-i}(\mathbf{Z}/p^i\mathbf{Z})^\times\\
            & 1
        \end{smallmatrix}\right]G_p^\circ$. This exhausts all left cosets present in the decomposition of $G_p^\circ s(\lambda) G_p^\circ$. Indeed, for $u\in (\mathbf{Z}/p^i\mathbf{Z})^\times$, we have the identities
        \begin{align}\left[\begin{smallmatrix}
            p^{\lambda-i}u^{-1} & \\
             & 1
        \end{smallmatrix}\right]p^{i-\lambda}\left[\begin{smallmatrix}
            p^i & u \\
            & p^{\lambda-i}
        \end{smallmatrix}\right]\left[\begin{smallmatrix}
            u & \\
             & 1
        \end{smallmatrix}\right]=n_0s(i)\  , \left[\begin{smallmatrix}
            p^{i-\lambda}u^{-1} & \\
             & 1
        \end{smallmatrix}\right]\left[\begin{smallmatrix}
            p^\lambda & p^{\lambda-i} u\\
            & 1
        \end{smallmatrix}\right]\left[\begin{smallmatrix}
            u & \\
             & 1
        \end{smallmatrix}\right]=n_0s(i).
        \end{align}
        This together with \Cref{Ds} give us the decomposition
        $$\scalemath{0.9}{\mathfrak{H}_p\mathfrak{G}_p^\circ(s(\lambda),1) \mathfrak{G}_p^\circ=\mathfrak{H}_p\gamma_{\lambda,0,0}\mathfrak{G}_p^\circ\ \sqcup\left(\bigsqcup_{0\leq i<\lambda}\mathfrak{H}_p\gamma_{i,\lambda-i,i-\lambda}\mathfrak{G}_p^\circ\right)  \sqcup\left(\bigsqcup_{0\leq i<\lambda} \mathfrak{H}_p\gamma_{i,i-\lambda,0}\mathfrak{G}_p^\circ\right).}$$
        In a similar fashion as before, it is not hard to see that on each of these double cosets, the function $\phi_{s(\lambda)^{-1}}*\xi_{0,0,0}$ takes the value $1$. In other words, we have shown that
        $$\phi_{s(\lambda)^{-1}}*\xi_{0,0,0}=\xi_{\lambda,0,0}+\sum_{0\leq i<\lambda}(\xi_{i,\lambda-i,i-\lambda}+\xi_{i,i- \lambda,0}).$$
        Thus, $\phi_{s(\lambda)^{-1}}$ always sends $\xi_{0,0,0}$ to an element of $\mathcal{A}_\lambda-\mathcal{A}_{\lambda-1}$. From here, the result follows from a straight forward argument, both at a non-integral and integral level.
    \end{proof}

\noindent 
\begin{lem}\label{lemmap}The linear isomorphism
$i:C_{\mathfrak{H}_p\backslash c}^{\infty,\mathrm{sph}}\simeq\mathbf{C}[\mathfrak{G}_p/\mathfrak{G}_p^\circ]_{\mathfrak{H}_p}$ given in \emph{\Cref{sec inv to coinv}}, is $\mathcal{R}^{\mathfrak{G}_p}$-euivariant.
    For $\lambda\in\mathbf{Z}_{\geq 0}$ and $a,b\in\mathbf{Z}$ we have 
    \begin{align*}
       i(\xi_{\lambda,a})&=\#H_Z(\mathbf{Z}/p^\lambda\mathbf{Z})\left[\ch\left(\gamma_{\lambda,a}\mathfrak{G}_p^\circ\right)\right]\ \text{ if }H\text{ is non-split}\\
       i(\xi_{\lambda,a,b})&=\#H_Z(\mathbf{Z}/p^\lambda\mathbf{Z})\left[\ch\left(\gamma_{\lambda,a,b}\mathfrak{G}_p^\circ\right)\right]\ \text{ if }H\text{ is split}
    \end{align*}
    \begin{proof}
    The equivariance property follows by the equivariance of the inverse to $i$ given by integrating over $\mathfrak{H}_p$ (\Cref{sec inv to coinv}). For the two equalities, it suffices to show that 
    $$\#H_Z(\mathbf{Z}/p^\lambda\mathbf{Z})=\begin{dcases*}
                 \vol_{\mathfrak{H}_p}(\mathfrak{H}_p\cap \gamma_{\lambda,a}\mathfrak{G}_p^\circ \gamma_{\lambda,a}^{-1})^{-1},\ H \text{ non-split}\\
                 \vol_{\mathfrak{H}_p}(\mathfrak{H}_p\cap \gamma_{\lambda,a,b}\mathfrak{G}_p^\circ \gamma_{\lambda,a,b}^{-1})^{-1},\ H\text{ split}.
                 \end{dcases*}$$
            This follows from the following identities
             $$
    s(\lambda)^{-1}\left[\begin{smallmatrix}
                     a & b \\
                     -bD & a
                 \end{smallmatrix}\right]s(\lambda)=\left[\begin{smallmatrix}
                     a & bp^{-\lambda}\\
                      bp^\lambda D & a 
                 \end{smallmatrix}\right]\ \ , \ \
s(\lambda)^{-1}n_0^{-1}\left[\begin{smallmatrix}
                     a & \\
                     & b
                 \end{smallmatrix}\right]n_0s(\lambda)=\left[\begin{smallmatrix}
                     a &(a-b) p^{-\lambda}\\
                     & b
                 \end{smallmatrix}\right].
             $$
    \end{proof}
\end{lem}

\begin{defn}\label{def level subgroup}
    For any $f\in\mathbf{Z}_{\geq 1}$, we define $H_p^\circ[p^f]$ to be the open compact subgroup of $H_p^\circ$ given by the pullback of $Z_G(\mathbf{Z}/p^f\mathbf{Z})\subseteq H(\mathbf{Z}/p^f\mathbf{Z})$ under the projection $H_p^\circ \rightarrow H(\mathbf{Z}/p^f\mathbf{Z})$. We also define $\mathfrak{G}_p^\circ[p^f]$ to be the open compact subgroup of $\mathfrak{G}_p^\circ$ given by $\GL_2(\mathbf{Z}_p)\times H_p^\circ[p^f].$
\end{defn}
\begin{rem}\label{volume of level subgroup}
    For $f\in\mathbf{Z}_{\geq 1}$, the volume of $H_p^\circ[p^f]$ is determined by the following formula
    $$\frac{1}{\vol_{H_p}(H_p^\circ[p^f])}=
        \#H_Z(\mathbf{Z}/p^f \mathbf{Z}).$$
\end{rem}

\begin{defn}\label{defnlattices}
    One can define $\mathbf{Z}[1/p]$-lattices in $\mathbf{C}[\mathfrak{G}_p/\mathfrak{G}_p^\circ]$, in the spirit of \cite[Definition $3.2.1$]{Loeffler_2021}.
    \begin{align*}
        \mathcal{H}^{H_p}&:=\mathrm{span}_{\mathbf{Z}[1/p]}\left\{\frac{1}{\vol_{\mathfrak{H}_p}(\mathfrak{H}_p\cap g\mathfrak{G}_p^\circ g^{-1})}\ch(g\mathfrak{G}_p^\circ)\ |\ g\in \mathfrak{G}_p\right\}\\
    \mathcal{H}_1^{H_p}&:=\mathrm{span}_{\mathbf{Z}[1/p]}\left\{\frac{1}{\vol_{\mathfrak{H}_p}(\mathfrak{H}_p\cap g\mathfrak{G}_p^\circ[p] g^{-1})}\ch(g\mathfrak{G}_p^\circ)\ |\ g\in \mathfrak{G}_p\right\}.
    \end{align*}
We denote their images in $\mathbf{C}[\mathfrak{G}_p/\mathfrak{G}_p^\circ]_{\mathfrak{H}_p}$ by $\mathcal{I}^{H_p}:=\mathcal{I}^{H_p}(\mathbf{Z}[1/p])$ and $\mathcal{I}_1^{H_p}:=\mathcal{I}_1^{H_p}(\mathbf{Z}[1/p])$ respectively.
\end{defn}
    \begin{lem}\label{lemlattices}
        Let $i$ be the Hecke equivariant map given in \Cref{sec inv to coinv}. Then the following are true:
        \begin{enumerate}
            \item The lattice $\mathcal{I}^{H_p}$ is the image under $i$, of $\mathcal{L}^{H_p}$
            \item The lattice $\mathcal{I}_1^{H_p}$ is the image under $i$, of $\mathcal{L}_1^{H_p}$.
        \end{enumerate}
        \begin{proof}
             The first part is immediate by construction. The second part follows from the definition of the lattice $\mathcal{L}_1^{H_p}$, the fact that $\vol_{H_p}(H_p^\circ[p])=\# H_Z(\mathbf{F}_p)$ and the fact that for $\lambda\geq 1$ we have
             \begin{align*}
                 \vol_{\mathfrak{H}_p}(\mathfrak{H}_p\cap \gamma_{\lambda,a}\mathfrak{G}_p^\circ \gamma_{\lambda,a}^{-1})^{-1}&=\vol_{\mathfrak{H}_p}(\mathfrak{H}_p\cap \gamma_{\lambda,a}\mathfrak{G}_p^\circ[p] \gamma_{\lambda,a}^{-1})^{-1},\ H \text{ non-split}\\
                 \vol_{\mathfrak{H}_p}(\mathfrak{H}_p\cap \gamma_{\lambda,a,b}\mathfrak{G}_p^\circ \gamma_{\lambda,a,b}^{-1})^{-1}&=\vol_{\mathfrak{H}_p}(\mathfrak{H}_p\cap \gamma_{\lambda,a,b}\mathfrak{G}_p^\circ[p] \gamma_{\lambda,a,b}^{-1})^{-1},\ H\text{ split}
                 \end{align*}
                 This once again follows from the matrix idenities given at the end of the proof of \Cref{lemmap}.
    \end{proof}
    \end{lem}
    
    \begin{cor}\label{cor Co and I free}
        The modules $\mathbf{C}[\mathfrak{G}_p/\mathfrak{G}_p^\circ]_{\mathfrak{H}_p}$ and $\mathcal{I}^{H_p}$ are free of rank one over $\mathcal{R}^{\mathfrak{G}_p}$ and $\mathcal{R}_{\mathbf{Z}[1/p]}^{\mathfrak{G}_p}$ respectively, generated by $[\ch(\mathfrak{G}_p^\circ)]$.
        \begin{proof}
            This follows at once from \Cref{thmCunivfree} and \Cref{lemlattices}.
        \end{proof}
    \end{cor}
\subsection{Multiplicity one}\label{sec mult 1}
 As mentioned in the introduction, the module-theoretic structure results of \Cref{UHM} give a new approach to a ``multiplicity $\leq 1$'' statement for a certain family of representations. Recall that $H_Z$ is defined to be the quotient $H/Z_G$. Firstly, we make the following definition, where unlike most places in the literature, we do not require irreducibility.
 \begin{defn}\label{def non-deg unr}
     A smooth $G_p$-representation $\pi_p$ is unramified if it contains a non-zero $G_p^\circ$-fixed vector. Additionally, it is non-degenerate if this vector is unique up to scalars, and generates $\pi_p$ as a $G_p$-representation. 
 \end{defn}
  \noindent \noindent As mentioned in the introduction, for zero characteristic and irreducible representations, the multiplicity one was proven in \cite{Waldspurger} and later \cite{saito1993tunnell}. For the reducible case, the multiplicity one can be deduced from the relatively recent work of \cite{moeglin2010conjecture} via $\mathrm{SO}_3\simeq \mathrm{PGL}_2$. In positive characteristic and $H$ split the multiplicity one can be found in \cite{minguez:hal-01176198}. For $H$ split, zero characteristic and irreducible, the existence of good test vectors is part of \cite{prasad1990trilinear}. We finally note that statements regarding equality of dimension in the $\ell$-modular case, can also be made, but we won't go into them. A lot of the proofs mentioned above are rather lengthy and indirect and distinguish between several cases. The proof given here provides a unified approach that takes cares of all unramified non-degenerate representations at once for both zero and positive characteristic.
   \begin{thm}\label{mult1}
       Let $\mathbf{F}$ be an algebraically closed field of characteristic which does not divide $p\cdot \#(H/Z_G)(\mathbf{F}_p)$. Let $\pi_p$ be an unramified non-degenerate (not necessarily irreducible) $\mathbf{F}$-linear $G_p$-representation and let $\chi_p$ be an unramified $\mathbf{F}$-linear character of $H_p$. Then
    $$\dim_\mathbf{F}\Hom_{\mathfrak{H}_p}(\pi_p\boxtimes \chi_p,\mathbf{ 1})\leq 1.$$
    Furthermore, every such linear form is completely determined by its value on the spherical vector.
   \end{thm}
     \begin{proof}
     We write $W_{\pi_p}^\mathrm{sph}$ for the spherical vector of $\pi_p$. 
        We firstly assume that $\mathbf{F}$ is of characteristic zero. We have a surjective linear map
         $$\mathbf{F}[\mathfrak{G}_p/\mathfrak{G}_p^\circ]\longrightarrow \pi_p\boxtimes \chi_p\ ,\ P\mapsto P\cdot (W_{\pi_p}^\mathrm{sph}\otimes 1)$$
         which descends to a surjective map on coinvariants
         $$\mathbf{F}[\mathfrak{G}_p/\mathfrak{G}_p^\circ]_{\mathfrak{H}_p}\longrightarrow (\pi_p\boxtimes \chi_p)_{\mathfrak{H}_p}\ ,\ [P]\mapsto [P\cdot (W_{\pi_p}^\mathrm{sph}\otimes 1)]. $$
         However, by \Cref{cor Co and I free}, the module $\mathbf{F}[\mathfrak{G}_p/\mathfrak{G}_p^\circ]_{\mathfrak{H}_p}$ is cyclic over $\mathcal{H}_{\mathfrak{G}_p}^\circ$ generated by $[\ch(\mathfrak{G}_p^\circ)]$. Thus, every one of its elements can be written as $\mathcal{Q}*[\ch(\mathfrak{G}_p^\circ)]=\left[\mathcal{Q}^{'}\right]$ for some $\mathcal{Q}\in \mathcal{H}_{\mathfrak{G}_p}^\circ$. Hence, every element of $(\pi_p\boxtimes \chi_p)_{\mathfrak{H}_p}$ is of the form $\left[\mathcal{Q}^{'}\cdot (W_{\pi_p}^\mathrm{sph}\otimes 1)\right]=\Theta_{\pi_p\boxtimes \chi_p}\left(\mathcal{Q}^{'}\right)[W_{\pi_p}^\mathrm{sph}\otimes 1]$, where $\Theta_{\pi_p\boxtimes \chi_p}$ denotes the spherical Hecke eigensystem of $\pi_p\boxtimes \chi_p$ (which exists by the non-degeneracy assumption). Finally, giving an element of $\Hom_{\mathfrak{H}_p}(\pi_p\boxtimes \chi_p,\mathbf{ 1})$ is the same as giving an element of $\Hom_\mathbf{F}((\pi_p\boxtimes \chi_p)_{\mathfrak{H}_p},\mathbf{1})$ and we have just seen that the latter is at most one dimensional and every element of the $\Hom$-space is determined by its value on $[W_{\pi_p}^\mathrm{sph}\otimes 1]$. This concludes the proof in the case of zero characteristic. We now explain why everything still works for $\mathbf{F}=\overline{\mathbf{F}}_\ell$ for a prime $\ell$ satisfying the conditions of the theorem. Consider the $\overline{\mathbf{F}}_\ell$-valued spherical Hecke algebra $\mathcal{H}_{\mathfrak{G}_p}^\circ(\overline{\mathbf{F}}_\ell)$. This is still a commutative algebra under convolution. Also consider the $\mathcal{H}_{\mathfrak{G}_p}^\circ(\overline{\mathbf{F}}_\ell)$-module $C_{\mathfrak{H}_p\backslash c}^{\infty,\mathrm{sph}}(\overline{\mathbf{F}}_\ell)$ defined as in \Cref{UHM} but with $\overline{\mathbf{F}}_\ell$-valued functions. By the integral part of \Cref{thmCunivfree}, $C_c^\infty(\mathfrak{H}_p\backslash\mathfrak{G}_p/\mathfrak{G}_p^\circ,\overline{\mathbf{Z}}_\ell)$ is cyclic over $\mathcal{H}_{\mathfrak{G}_p}^\circ(\overline{\mathbf{Z}}_\ell)$. Hence $C_{\mathfrak{H}_p\backslash c}^{\infty,\mathrm{sph}}(\overline{\mathbf{F}}_\ell)$ is cyclic over $\mathcal{H}_{\mathfrak{G}_p}^\circ(\overline{\mathbf{F}}_\ell)$. Now, by our assumption on $\ell\nmid\left( p\cdot \#H_Z(\mathbf{F}_p)\right)$ and \Cref{lemmap}, the Hecke equivariant morphism $i$ is non-zero and hence still surjective onto $\overline{\mathbf{F}}_\ell[\mathfrak{G}_p/\mathfrak{G}_p^\circ]_{\mathfrak{H}_p}$. Thus the latter is still cyclic over $\mathcal{H}_{\mathfrak{G}_p}^\circ(\overline{\mathbf{F}}_\ell)$. The proof then proceeds in the exact same manner as before.
         \end{proof}
\begin{rem} In the complex case, one has equality of dimension if and only if $\omega_{\pi_p}=\chi_p|_{\mathbf{Q}_p^\times}$. We will go over this in more detail in \Cref{seczeta} where we will explicitly write down non-zero elements of $\Hom_{\mathfrak{H}_p}(\pi_p\boxtimes \chi_p,\mathbf{C})$. Similar statements can be made in the $\ell$-modular case, but we will not be concerned with this from now on. In the complex setting, it is also well known that the unramified irreducible principal series and the unramified twists of $\mathrm{n}\text{-}\mathrm{Ind}_{B_p}^{G_p}\delta_B^{1/2}$ exhaust all isomorphism classes of unramified non-degenerate representations of $G_p$. In the modular setting, it is still true that the irreducible unramified principal series are non-degenerate. This is part of Theorem 
 $0.1$ in \cite{minguez:hal-01176198} but for $\GL_2$ it has been known before that as stated in \cite{barthel1995modular}. For reducible principal series in the modular setting, it is a bit more delicate. See for example \cite{vigneras1989representations}.
 \end{rem}

 \subsection{Universal Hecke operators}\label{sec unv hecke opers}
  Given any smooth $\mathfrak{G}_p$-representation $V$, and any $(\mathfrak{H}_p\times \mathfrak{G}_p)$-equivariant map $\zeta$ from $\mathcal{H}(\mathfrak{G}_p)$ to $V$, recall the induced $\mathfrak{G}_p$-equivariant map $\mathfrak{z}:C_c^\infty(\mathfrak{H}_p\backslash\mathfrak{G}_p)\rightarrow V$ as in \Cref{sec inv to coinv}. Also note that we have a trace map $$\mathrm{Tr}:=\mathrm{Tr}^{\mathfrak{G}_p^\circ[p]}_{\mathfrak{G}_p^\circ}:C_c^\infty(\mathfrak{H}_p\backslash\mathfrak{G}_p/\mathfrak{G}_p^\circ[p])\rightarrow C_{\mathfrak{H}_p\backslash c}^{\infty,\mathrm{sph}},\ \xi_1\mapsto\sum_{k\in\mathfrak{G}_p^\circ/\mathfrak{G}_p^\circ[p]}k\cdot\xi_1.$$

\begin{defn}\label{def unv hecke ops}
    Let $\xi$ be an element of $C_{\mathfrak{H}_p\backslash c}^{\infty,\mathrm{sph}}$. The universal Hecke operator attached to $\xi$, which we denote by $\mathcal{P}_\xi^\mathrm{unv}$, is the unique element of $\mathcal{R}^{\mathfrak{G}_p}$ that satisfies $\mathcal{P}_\xi^\mathrm{unv}*\ch(\mathfrak{H}_p\mathfrak{G}_p^\circ)=\xi$.
\end{defn}

\begin{thm}\label{thmL1}
\begin{enumerate}
\item Let $\xi\in \mathcal{L}^{H_p}.$ Then 
$$\mathfrak{z}(\xi)=\mathcal{P}_\xi^\mathrm{unv}\cdot\mathfrak{z}(\ch(\mathfrak{H}_p\mathfrak{G}_p^\circ))\ \ \mathrm{with}\ \ \mathcal{P}_\xi^\mathrm{unv}\in\mathcal{R}_{\mathbf{Z}[1/p]}^{\mathfrak{G}_p}.$$
   \item Let $\xi_1\in C_c^\infty(\mathfrak{H}_p\backslash\mathfrak{G}_p/\mathfrak{G}_p^\circ[p],\mathbf{Z}[1/p])$. Then $\mathrm{Tr}(\xi_1)\in\mathcal{L}_1^{H_p}$ and 
    $$\mathrm{norm}^{\mathfrak{G}_p^\circ[p]}_{\mathfrak{G}_p^\circ}\mathfrak{z}(\xi_1)=\mathcal{P}_{\mathrm{Tr}(\xi_1)}^\mathrm{unv}\cdot \mathfrak{z}(\ch(\mathfrak{H}_p\mathfrak{G}_p^\circ)).$$
    \end{enumerate}
    \begin{proof} The equality in the first part follows abstractly from \Cref{def unv hecke ops}, the $\mathcal{R}^{\mathfrak{G}_p}$-equivariance of the isomorphism $C_{\mathfrak{H}_p\backslash c}^{\infty,\mathrm{sph}}\simeq \mathbf{C}[\mathfrak{G}_p/\mathfrak{G}_p^\circ]_{\mathfrak{H}_p}$ and \Cref{equiv}. The fact that $\mathcal{P}_\xi^\mathrm{unv}\in\mathcal{R}_{\mathbf{Z}[1/p]}^{\mathfrak{G}_p}$ for $\xi\in\mathcal{L}^{H_p}$ is \Cref{thmCunivfree}. For the second part,
        to see that $\mathrm{Tr}(\xi_1)\in \mathcal{L}_1^{H_p}$ it suffices to check that it takes values in $\# H_Z(\mathbf{F}_p)\mathbf{Z}[1/p]$ on $H_p\times \{1\}.$ Indeed, for $h\in H_p$, we have $\mathrm{Tr}(\xi_1)((h,1))=\sum_{k\in\mathfrak{G}_p^\circ/\mathfrak{G}_p^\circ[p]}\xi_1((h,1)k)=[H_p:H_p[p]]\xi_1((h,1))$ and the claim follows by \Cref{volume of level subgroup}. For the norm-relations, we note that by the $\mathfrak{G}_p$-equivariance of $\mathfrak{z}$, $\mathrm{norm}^{\mathfrak{G}_p^\circ[p]}_{\mathfrak{G}_p^\circ}\mathfrak{z}(\xi_1)=\mathfrak{z}(\mathrm{Tr}(\xi_1))$ and this is the same as $\mathfrak{z}(\mathcal{P}_{\mathrm{Tr}(\xi_1)}^\mathrm{unv}\cdot\ch(\mathfrak{H}_p\mathfrak{G}_p^\circ))$ by \Cref{def unv hecke ops}. The result once again follows by equivariance and \Cref{equiv}.
    \end{proof}
\end{thm}
 Thus, to prove \Cref{thmintro1} it remainds to show that the universal Hecke operators attached to integral test data in the lattice $\mathcal{L}_1^{H_p}$, generate the ideal
$$\mathfrak{h}^{H_p}:=\begin{dcases*}
\left\langle\#H_Z(\mathbf{F}_p),\mathcal{T}_p\right\rangle\ ,\ H\text{ non-split}\\
\left\langle\#H_Z(\mathbf{F}_p),{\mathcal{P}_p^{H}}'(p^{-1/2})\right\rangle\ ,\ H\text{split}
\end{dcases*}$$
of $\mathcal{R}_{\mathbf{Z}[1/p]}^{\mathfrak{G}_p}$. See \Cref{seczeta} for the definition of $\mathcal{P}_p^{H}$. 
In the non-split case, we can show this using a technical lemma regarding the behavior of certain double coset operators in the Hecke algebra of $G_p$. In the split case, this is not sufficient and the most natural approach to this is through studying Shintani functions, and constructing a so-called universal Shintani function, which will carry information about the universal Hecke operator $\mathcal{P}_\xi^{\mathrm{unv}}$.

\subsubsection{Modules for a fixed character}\label{sec modules for a fixed character} Given a smooth $G_p$ representation $V$ and an unramified character of $H_p$, one can define the notion of $(H_p\times G_p;\chi_p)$-equivariant maps $\mathcal{H}(G_p)\rightarrow V$. These are equivariant maps where $G_p$ acts on the left through right translations and on the right by its assigned action on $V$, and $H_p$ acts on the left through inverse left translations and on the right by $\chi_p$. It is clear that, any such map factors through $$\mathbf{C}[G_p/G_p^\circ]_{H_p,\chi_p}:=\mathbf{C}[G_p/G_p^\circ]/\langle h\cdot f-\chi_p(h)f\ |\ h\in H_p, f\in\mathbf{C}[G_p/G_p^\circ]\rangle.$$This is naturally a module over $\mathcal{R}^{\mathfrak{G}_p}$ via $(\phi\otimes f)\cdot [\xi]:=\int_{G_p}\int_{H_p}\phi(g)f(h)(h^{-1},g)\cdot [\xi]\ dg\ dh$. By \Cref{cor Co and I free}, we can now define a canonical Hecke equivariant map of $\mathcal{R}^{\mathfrak{G}_p}$-modules 
$$\mathrm{Pr}_{\chi_p}^{H_p,\mathrm{unv}}:\mathbf{C}[\mathfrak{G}_p/\mathfrak{G}_p^\circ]_{\mathfrak{H}_p}\longrightarrow\mathbf{C}[G_p/G_p^\circ]_{H_p,\chi_p}\ ,\ [\ch(\mathfrak{G}_p^\circ)]\mapsto [\ch(G_p^\circ)].$$
It is a natural question to ask whether the module of $\chi_p$-twisted coinvariants is itself free over some quotient of $\mathcal{R}^{\mathfrak{G}_p}$. It is clear by construction that $\mathbf{C}[G_p/G_p^\circ]_{H_p,\chi_p}$ is naturally a module over the quotient 
$$\mathcal{R}_{\chi_p}^{\mathfrak{G}_p}:=\begin{cases}
    \mathcal{R}^{\mathfrak{G}_p}/\left(\mathcal{S}_p-\chi_p(p)^{-1}\right)\ ,\ H\text{ non-split}\\
    \mathcal{R}^{\mathfrak{G}_p}/\left(A_p-\chi_p\left(\left[\begin{smallmatrix}
        p & \\
        & 1
    \end{smallmatrix}\right]\right)^{-1},B_p-\chi_p\left(\left[\begin{smallmatrix}
        1 & \\
        & p
    \end{smallmatrix}\right]\right)^{-1}\right)\ ,\ H\text{ split}.
\end{cases}$$

\noindent
\begin{lem}\label{lemchicyc}
    The Hecke equivariant map $\mathrm{Pr}_{\chi_p}^{H_p,\mathrm{unv}}$ is surjective.
    \begin{proof}
        The image of $\mathrm{Pr}_{\chi_p}^{H_p,\mathrm{unv}}$ is completely determined from the submodule traced out by the elements $\phi_{s(\lambda)^{-1}}\cdot  [\ch(G_p^\circ)]$, where recall that $\phi_{s(\lambda)^{-1}}$ denotes the characteristic function of the double coset $G_p^\circ s(\lambda)^{-1}G_p^\circ$. It is once again clear by construction that $\phi_{s(\lambda)^{-1}}\cdot  [\ch(G_p^\circ)]=[\phi_{s(\lambda)}]$ in $\mathbf{C}[G_p/G_p^\circ]_{H_p,\chi_p}$. Looking at the proof of \Cref{thmCunivfree}, if $H$ is non-split, then practically nothing changes in $(1)$, by our assumption that $\chi_p$ is unramified. Thus, if $H$ is non-split, we have 
        $$\phi_{s(\lambda)^{-1}}\cdot [\ch(G_p^\circ)]=\#H_Z(\mathbf{Z}/p^\lambda\mathbf{Z})[\ch(s(\lambda) G_p^\circ)]$$
        in $\mathbf{C}[G_p/G_p^\circ]_{H_p,\chi_p}$. Thus, from \Cref{Ds}, this implies surjectivity of $\mathrm{Pr}_{\chi_p}^{H_p,\mathrm{unv}}$. If $H$ is split, then looking at $(2)$ in the proof of \Cref{thmCunivfree} we can deduce that $\phi_{s(\lambda)^{-1}}\cdot  [\ch(G_p^\circ)]$ is given by
        $$\varphi(p^\lambda)[\ch(n_0s(\lambda) G_p^\circ)]+\sum_{0\leq i<\lambda}\varphi(p^i)\left(\chi_p\left(\left[\begin{smallmatrix}
            1 & \\
             & p^{i-\lambda}
        \end{smallmatrix}\right]\right)+\chi_p\left(\left[\begin{smallmatrix}
            p^{i-\lambda} & \\
             & 1
        \end{smallmatrix}\right]\right)\right)[\ch(n_0s(i)G_p^\circ)]$$
        in $\mathbf{C}[G_p/G_p^\circ]_{H_p,\chi_p}$. From this, using \Cref{Ds} and induction on $\lambda\in\mathbf{Z}_{\geq 0}$ shows surjectivity of $\mathrm{Pr}_{\chi_p}^{H_p,\mathrm{unv}}$ in the split case.
    \end{proof}
\end{lem}
\begin{rem}
    It is not hard to see that the map $\mathrm{Pr}_{\chi_p}^{H_p,\mathrm{unv}}$ is given explicitly by $$\mathrm{Pr}_{\chi_p}^{H_p,\mathrm{unv}}\left([\ch((g,h)\mathfrak{G}_p^\circ)]\right)=\chi_p(h)[\ch(gG_p^\circ)].$$
    This follows from the proof of \Cref{lemchicyc} and the fact that $\ch(h H_p^\circ)*[\ch(\mathfrak{G}_p^\circ)]=[\ch((1,h^{-1})\mathfrak{G}_p^\circ)].$
\end{rem}
\begin{cor}\label{cor mod of fixed char is free}
    The module $\mathbf{C}[G_p/G_p^\circ]_{H_p,\chi_p}$ is free of rank one over $\mathcal{R}_{\chi_p}^{\mathfrak{G}_p}$, generated by $[\ch(G_p^\circ)]$.
    \begin{proof}
        By \Cref{cor Co and I free} it is clear that $\mathbf{C}[G_p/G_p^\circ]_{H_p,\chi_p}$ is cyclic over $\mathcal{R}_{\chi_p}^{\mathfrak{G}_p}$ generated by $[\ch(G_p^\circ)]$. To show freeness, one argues in the same way as \Cref{thmCunivfree}.
    \end{proof}
\end{cor}

\subsubsection{Universal property}\label{sec univ prop} From \Cref{cor mod of fixed char is free} we know that $\mathbf{C}[G_p/G_p^\circ]_{H_p,\chi_p}$ is free of rank one over $\mathcal{R}_{\chi_p}^{\mathfrak{G}_p}$. For an element $[\mu]\in \mathbf{C}[G_p/G_p^\circ]_{H_p,\chi_p}$, we write $\mathcal{P}_{[\mu],\chi_p}$ for the unique element of $\mathcal{R}_{\chi_p}^{\mathfrak{G}_p}$ such that  $[\mu]$ is equal to $\mathcal{P}_{[\mu],\chi_p}\cdot [\ch(G_p^\circ)].$ Finally, we write $\mathrm{pr}_{\chi_p}$ for the projection $\mathcal{R}^{\mathfrak{G}_p}\rightarrow\mathcal{R}_{\chi_p}^{\mathfrak{G}_p}$.\\
\\ Let $\xi\in C_{\mathfrak{H}_p\backslash c}^{\infty,\mathrm{sph}}$. The Hecke operator $\mathcal{P}_\xi^\mathrm{unv}$ satisfies the following universal property:  
For every unramified $\mathbf{C}$-linear character $\chi_p$ of $H_p$, we have
$$\mathrm{pr}_{\chi_p}(\mathcal{P}_\xi^\mathrm{unv})=\mathcal{P}_{[\mu_\xi],\chi_p}$$
where $[\mu_\xi]:=\mathrm{Pr}_{\chi_p}^{H_p,\mathrm{unv}}(i(\xi))$, and this uniquely determines the Hecke operator $\mathcal{P}_\xi^\mathrm{unv}$.\\
\subsubsection{Specialisation at an unramified character} Let $V$ be a smooth $G_p$-representation, $\chi_p$ an unramified character of $H_p$ and $\zeta_{\chi_p}$ a $(H_p\times G_p;\chi_p)$-equivariant map. We can regard $V$ as a smooth $\mathfrak{G}_p$-representation where the $H_p$ factor acts through $\chi_p$. One can check that the map
    \begin{align}\label{eq:specialisation}
    \mathcal{H}(\mathfrak{G}_p)\rightarrow V\ ,\ \phi\otimes f\mapsto f\cdot \zeta_{\chi_p}\left(\phi\right)
    \end{align}
    is $(\mathfrak{H}_p\times \mathfrak{G}_p)$-equivariant and thus one can apply the result of \Cref{sec unv hecke opers}. This is the setup that will usually occur in practice, with $H_p$ acting through an unramified character. Using the extended group $\mathfrak{G}$ and the theory developed in that setting, we have constructed a space of test data that is independent of $\chi_p$, and gives rise to Hecke operators that are also independent of $\chi_p$, but still control the image of test data under \eqref{eq:specialisation}, upon specialising at any fixed character $\chi_p$.

\section{Zeta integrals}\label{seczeta}

Before introducing zeta integrals, we provide some necessary background. We fix once and for all an additive character $\psi:\mathbf{Q}_p\rightarrow\mathbf{C}^\times$ of conductor $(1)$. Let $\sigma_{p,1}$ and $\sigma_{p,2}$ be two characters of $\mathbf{Q}_p^\times$. From now on,  we will write $\mathcal{B}(\sigma_{p,1},\sigma_{p,2})$ for the principal-series representation given by the normalised induction $\mathrm{n}\text{-}\mathrm{Ind}_{B_p}^{G_p}\sigma_p$. Here $B_p$ is the upper triangular Borel, and the character $\left[\begin{smallmatrix}
    \sigma_{p,1} & \\
     & \sigma_{p,2}
\end{smallmatrix}\right]$ is extended to a character $\sigma_p$ of $B_p$ in the usual way, by acting trivially on the unipotent radical $N_p\subseteq B_p$. It is a well known fact that such a representation is irreducible if and only if $\sigma_{p,1}\sigma_{p,2}^{-1}\neq |\cdot |_p^{\pm 1}$ (\cite{bump_1997} Section $4.5$). It is clearly unramified if and only if $\sigma_{p,1}$ and $\sigma_{p,2}$ are unramified, and in this case we have a canonical spherical vector given by $W_{\pi_p}^\mathrm{sph}(g)=\delta_{B}^{1/2}(b)\sigma_p(b)$ for $g=bk\in B_pG_p^\circ$. Here $\delta_B$ denotes the modular character of the non-unimodular group $B_p$. It is given on the diagonal torus by $\delta_B(t)=|\det(\mathrm{Ad}(t)|_\mathfrak{n})|_p$ where $\mathrm{Ad}$ denotes the adjoint action of $t\in T_p$ on the Lie algebra of upper triangular matrices. The action is then extended to $B_p$ in the usual way. \\
\\
\noindent Given an unramified principal-series $\pi_p=\mathcal{B}(\sigma_{p,1},\sigma_{p,2})$, we will denote its Satake parameters by $\alpha^{\pi_p}:=\sigma_{p,1}(p)$ and $\beta^{\pi_p}:=\sigma_{p,2}(p)$. We write $\Theta_{\pi_p}$ for the spherical Hecke eigensystem of $\pi_p$. It is well known that $\Theta_{\pi_p}(\mathcal{S}_p)=\alpha^{\pi_p}\beta^{\pi_p}$ and $\Theta_{\pi_p}(\mathcal{T}_p)=p^{1/2}(\alpha^{\pi_p}+\beta^{\pi_p})$ (\cite{bump_1997} Proposition $4.6.6$). We denote the $L$-factor attached to the representation $\pi_p$ by $L(\pi_p,s)$, which is given by the rational function 
$$L(\pi_p,s):=(1-\alpha^{\pi_p}p^{-s})^{-1}(1-\beta^{\pi_p}p^{-s})^{-1}.$$
We will be mainly considering unramified $\mathfrak{G}_p$-representations of the form $\pi_p\boxtimes \chi_p$ where $\pi_p$ is an unramified principal-series of $G_p$ and $\chi_p$ is an unramified character of $H_p$. For such a representation, we will denote its normalized spherical vector by $W_{\pi_p}^\mathrm{sph}\otimes 1$. We write $\Theta_{\pi_p\boxtimes \chi_p}$ for its spherical Hecke eigensystem which factors through $\mathcal{R}^{\mathfrak{G}_p}$ whenever $\omega_{\pi_p}=\chi_p|_{\mathbf{Q}_p^\times}$. We also attach the following $L$-factor to $\pi_p\boxtimes \chi_p$, given by 
\begin{align*}
L^{H_p}(\pi_p\boxtimes\chi_p,s):=\begin{cases}
    L(\pi_p,s)\ ,\ H\text{ non-split}\\
   (1-\alpha^{\pi_p}\chi_{p}(\left[\begin{smallmatrix}
       p & \\
        & 1
   \end{smallmatrix}\right])^{-1}p^{-s})^{-1}(1-\beta^{\pi_p}\chi_{p}(\left[\begin{smallmatrix}
       p & \\
        & 1
   \end{smallmatrix}\right])^{-1}p^{-s})^{-1}\ ,\ H\text{ split}.
\end{cases}
\end{align*}

\begin{lem}\label{lem euler factor}
    Let $\pi_p$ be an unramified principal-series of $G_p$ and $\chi_p$ an unramified character of $H_p$ with $\omega_{\pi_p}=\chi_p|_{\mathbf{Q}_p^\times}$. There exists a degree $2$ polynomial $\mathcal{P}_p^{H}$ in $\mathcal{R}^{\mathfrak{G}_p}[X]$ for which we have
    $$\Theta_{\pi_p\boxtimes \chi_p}(\mathcal{P}_p^{H})(p^{-s})=L^{H_p}(\pi_p\boxtimes \chi_p,s)^{-1}.$$
    \begin{proof}
        The polynomial $\mathcal{P}_p^{H}$ is explicitly given by 
        $$\mathcal{P}_p^{H}(X):=\begin{cases}
            1-p^{-1/2} \mathcal{T}_pX+\mathcal{S}_pX^2\ ,\ H\text{ non-split}\\
            1-p^{-1/2}\mathcal{T}_pA_p^{-1} X+\mathcal{S}_pA_p^{-2}X^2\ ,\ H\text{ split}.
        \end{cases}$$
        The result follows from a direct computation.
    \end{proof}
    \end{lem}
 \begin{defn}\label{def whit}
     An admissible representation $\pi_p$ of $G_p$ is of Whittaker type if it is either irreducible and generic, or a reducible principal-series with a one-dimensional quotient.
 \end{defn}
 \noindent If $\pi_p$ is of Whittaker type, the it has a well-defined Whittaker model. As remarked in \cite{loeffler2021gross} the irreducible unramified principal-series together with the unramified twists of  $\mathrm{n}\text{-}\mathrm{Ind}_{B_p}^{G_p}\delta_B^{1/2}$, exhaust all isomorphism classes of unramified $G_p$-representations of Whittaker type.
 We will often identify such a principal-series with its Whittaker model. In that case we will denote its normalized spherical vector by $W_{\pi_p}^\mathrm{sph}$, which satisfies $W_{\pi_p}^\mathrm{sph}(1)=1.$ It will also be convenient notation-wise to identify an unramified character $\chi_p$ of $H_p$, with the space $\{z\cdot \chi_p:H_p\rightarrow \mathbf{C}\ |\ z\in\mathbf{C}\}$ where the $H_p$-action is given by right translations. By abuse of notation, we denote this space by $\mathcal{W}(\chi_p)$.
 \begin{defn}\label{whitzeta}
     Let $H$ be non-split. For an unramified $G_p$-representation $\pi_p$ of Whittaker type , and an unramified character $\chi_p$ of $H_p$ with $\omega_{\pi_p}=\chi_p|_{\mathbf{Q}_p^\times}$, we define the zeta integral $$\mathfrak{z}_p^{ns}:\mathcal{W}(\pi_p)\boxtimes \mathcal{W}(\chi_p)\longrightarrow \mathbf{C}\ ,\ \mathfrak{z}_p^{ns}(W_1\otimes W_2):=W_2(1)\int_{H_p^\circ} W_1(h)\ dh.$$
    
 \end{defn}
 
    \noindent The linear form defined above is just a finite sum of Whittaker functions. Furthermore, since  $H_p=Z_pH_p^\circ$, it is easy to see that the linear form $\mathfrak{z}_p^{ns}$, is in fact an element of $\Hom_{\mathfrak{H}_p}(\pi_p\boxtimes \chi_p,\mathbf{C})$, and it satisfies $\mathfrak{z}_p^{ns}(W_{\pi_p}^\mathrm{sph}\otimes 1)=1.$ The existence of this element, shows that $[W_{\pi_p}^\mathrm{sph}\otimes 1]$ is indeed non-zero in $(\pi_p\boxtimes \chi_p)_{\mathfrak{H}_p}$ for $H$ non-split.
 
 \begin{defn}
     Let $(\pi_p,\chi_p)$ be as in \emph{\Cref{whitzeta}}. We define the following zeta integral, similar to \cite{jacquet2006automorphic}.
     $$Z^{H_p}(W_1\otimes W_2,s):=\begin{dcases*}
         W_2(1)\int_{\mathbf{Q}_p^\times} W_1(\left[\begin{smallmatrix}
             x & \\
              & 1
         \end{smallmatrix}\right])|x|_p^{s-1/2}\ d^\times x\ ,\ H\text{ non-split}\\
         \int_{\mathbf{Q}_p^\times}W_1(\left[\begin{smallmatrix}
             x & \\
              & 1
         \end{smallmatrix}\right])W_2(\left[\begin{smallmatrix}
             x^{-1} & \\
             & 1
         \end{smallmatrix}\right])|x|_p^{s-1/2}\ d^\times x\ ,\ H\text{ split}.
     \end{dcases*}$$
     for $W_1\otimes W_2\in\mathcal{W}(\pi_p)\boxtimes \mathcal{W}(\chi_p)$, where $d^\times x$ is the unique Haar measure on $\mathbf{Q}_p^\times$ that gives $\mathbf{Z}_p^\times$ volume one. 
 \end{defn}
 \begin{prop}\label{propJPS} The integral $Z^{H_p}(W_1\otimes W_2,s)$ converges absolutely for $\mathrm{re}(s)\gg 0$ and has unique meromorphic continuation as a rational function of $p^s$. For varying $W_1\otimes W_2\in\mathcal{W}(\pi_p)\boxtimes \mathcal{W}(\chi_p)$, the elements $Z^{H_p}(W_1\otimes W_2,s)$ generate the fractional ideal $L^{H_p}(\pi_p\boxtimes \chi_p,s)\mathbf{C}[p^{ s},p^{-s}].$
     \begin{proof}
         This is a well-known result and a detailed proof in greater generality can be found in \cite{jacquet1983rankin} Section $2.7$.
     \end{proof}
 \end{prop}
 \begin{lem}\label{lemgro}
     Let $\pi_p\boxtimes \chi_p$ be as above. Then, $Z^{H_p}(W_{\pi_p}^\mathrm{sph}\otimes 1)=L^{H_p}(\pi_p\boxtimes \chi_p,s)$.
     \begin{proof}
         This is again quite well-known and follows from a direct computation (see for example $2.2.2$ in \cite{grossi2020norm}). The key ingredient is to express the values of the normalized Whittaker function $W_{\pi_p}^\mathrm{sph}$ on diagonal matrices, as Weyl-group invariant polynomials in the Satake parameters, in the style of \cite{shintani1976explicit}.
     \end{proof}
     \end{lem}
\begin{defn}\label{defzetaavg}
    Let $H$ be non-split. We can define another zeta integral by averaging $Z^{H_p}$ over $H_p^\circ$. We define 
    $$Z_\mathrm{avg}^{H_p}(W_1\otimes W_2,s):=\int_{H_p^\circ} Z^{H_p}(h\cdot (W_1\otimes W_2),s)\ dh\ ,\ W_1\otimes W_2\in\mathcal{W}(\pi_p)\boxtimes \mathcal{W}(\chi_p).$$
\end{defn}
\begin{rem}\label{remavg}
    By \Cref{lemgro} we have $Z_\mathrm{avg}^{H_p}(W_{\pi_p}^\mathrm{sph}\otimes 1,s)=Z^{H_p}(W_{\pi_p}^\mathrm{sph}\otimes 1,s)=L^{H_p}(\pi_p\boxtimes \chi_p,s)$. Also by \Cref{propJPS}, $Z_\mathrm{avg}^{H_p}(W_1\otimes W_2,s)$ also defines an element of the fractional ideal $L^{H_p}(\pi_p\boxtimes\chi_p,s)\mathbf{C}[p^{s},p^{-s}]$.
\end{rem}
\begin{defn}\label{def linear form}
    We define the linear form $\mathcal{Z}^{H_p}:\pi_p\boxtimes \chi_p\rightarrow\mathbf{C}$ by
    $$\mathcal{Z}^{H_p}(W_1\otimes W_2):= \lim_{s\rightarrow 0}\begin{dcases*}
       Z_\mathrm{avg}^{H_p}(W_1\otimes W_2,s+1/2)L^{H_p}(\pi_p\boxtimes \chi_p,s+1/2)^{-1}\ ,\ H\text{ non-split}\\
        Z^{H_p}(W_1\otimes W_2,s+1/2)L^{H_p}(\pi_p\boxtimes \chi_p,s+1/2)^{-1}\ ,\ H\text{ split}.
    \end{dcases*}$$
\end{defn}

    \noindent By \Cref{propJPS} and \Cref{remavg}, the linear form $\mathcal{Z}^{H_p}$ is well-defined both when $H$ is split and non-split. Recall that we are considering representations $\pi_p\boxtimes \chi_p$ where $\omega_{\pi_p}=\chi_p|_{\mathbf{Q}_p^\times}$. Under this assumption, a zeta integral computation, shows that in both cases, $\mathcal{Z}^{H_p}$ is in fact a non-zero element of $\Hom_{\mathfrak{H}_p}(\pi_p\boxtimes \chi_p,\mathbf{1})$, which satisfies $\mathcal{Z}^{H_p}(W_{\pi_p}^\mathrm{sph}\otimes 1)=1$. We provide the necessary details for this in the split case: Let $h=\left[\begin{smallmatrix}
        h_1 & \\
        & h_2
    \end{smallmatrix}\right]\in H_p$. We have
    \begin{align*}
        Z^{H_p}(h\cdot (W_1\otimes W_2),s+1/2)&=\int_{\mathbf{Q}_p^\times }(h\cdot W_1)(\left[\begin{smallmatrix}
            x & \\
            & 1
        \end{smallmatrix}\right])(h^{-1}\cdot W_2)(\left[\begin{smallmatrix}
            x^{-1} & \\
            & 1
        \end{smallmatrix}\right])|x|_p^s\ d^\times x\\
        &\scalemath{0.95}{=\omega_{\pi_p}(\left[\begin{smallmatrix}
            h_2 & \\
            & h_2
        \end{smallmatrix}\right])\chi_p(\left[\begin{smallmatrix}
            h_2 & \\
            & h_2
        \end{smallmatrix}\right])^{-1}\int_{\mathbf{Q}_p^\times} W_1\left(\left[\begin{smallmatrix}
            x h_1h_2^{-1} & \\
            & 1
        \end{smallmatrix}\right]\right) W_2\left(\left[\begin{smallmatrix}
            (x h_1h_2^{-1})^{-1} & \\
            & 1
        \end{smallmatrix}\right]\right) |x|_p^s \ d^\times x}\\
        &=|xh_1^{-1}h_2|_p^s\ Z^{H_p}(W_1\otimes W_2 ,s+1/2).
    \end{align*}
    Thus, applying $\lim_{s\rightarrow 0}\frac{(-)}{L^{H_p}(\pi_p\boxtimes \chi_p,s+1/2)}$ we get that $\mathcal{Z}^{H_p}(h\cdot (W_1\otimes W_2),s+1/2)=\mathcal{Z}^{H_p}(W_1\otimes W_2,s+1/2)$ as claimed.
    \\

    \noindent From this, we can deduce two things. Firstly, by \Cref{mult1}, the linear form $\mathcal{Z}^{H_p}$ forms a basis of this $\Hom$-space and shows that $[W_{\pi_p}^\mathrm{sph}\otimes 1]$ is non-zero in $(\pi_p\boxtimes \chi_p)_{\mathfrak{H}_p}$ in the split case well. Furthermore, if $H$ is non-split, \Cref{mult1} implies that $\mathcal{Z}^{H_p}$ and $\mathfrak{z}_p^{ns}$ are actually equal.
\section{Integral multiplicity one}\label{sec int structures}
Throughput, we let $\pi_p$ be an unramified $G_p$-representation of Whittaker type and $\chi_p$ an unramified character of $H_p$ with identical central characters. We also let $A$ denote a $\mathbf{Z}[1/p]$-algebra and we define the $A$-lattice inside $(\pi_p\boxtimes\chi_p)_{\mathfrak{H}_p}$ as follows:
$$\mathcal{L}^{H_p}_{\pi_p\boxtimes\chi_p}(A):=\mathrm{span}_{A}\left\{\frac{1}{\vol_{\mathfrak{H}_p}(\mathfrak{H}_p\cap g\mathfrak{G}_p^\circ g^{-1})}g\cdot W_{\pi_p\boxtimes\chi_p}^\mathrm{sph}\ |\ g\in\mathfrak{G}_p\right\}\subseteq(\pi_p\boxtimes\chi_p)_{\mathfrak{H}_p}.$$
A priori, $\mathcal{L}^{H_p}_{\pi_p\boxtimes\chi_p}(A)$ is simply an $A$-submodule of the one-dimensional vector space $(\pi_p\boxtimes\chi_p)_{\mathfrak{H}_p}$ and hence has possibly uncountable $A$-rank. However, we have the following.
\begin{thm}\label{integral mult one}
    Suppose that the spherical Hecke eigensystem of $\pi_p\boxtimes\chi_p$ restricts to a morphism $\mathcal{R}_A^{\mathfrak{G}_p}\rightarrow A$. Then the following equivalent statements are true:
    \begin{enumerate}
        \item  The unique normalized period $\mathcal{Z}^{H_p}$ in $\mathrm{Hom}_{\mathfrak{H}_p}(\pi_p\boxtimes\chi_p,\mathbf{1})$ takes values in $A$ when evaluated on elements of $\mathcal{L}^{H_p}_{\pi_p\boxtimes\chi_p}(A)$.
        \item We have $\mathrm{rank}_A\ \mathcal{L}^{H_p}_{\pi_p\boxtimes\chi_p}(A)=1.$
    \end{enumerate}
    \begin{proof}
        This follows at once from \Cref{lemlattices} and the first part of \Cref{thmL1} by taking $V=(\pi_p\boxtimes\chi_p)^\vee$ and $\mathfrak{z}$ to be the map given by the composition of $i:C_{\mathfrak{H}_p\backslash c}^{\infty}\simeq \mathcal{H}(\mathfrak{G}_p)_{\mathfrak{H}_p}$ and $\mathcal{H}(\mathfrak{G}_p)_{\mathfrak{H}_p}\rightarrow(\pi_p\boxtimes\chi_p)^\vee,\ \xi\mapsto \left\{ W\mapsto \mathcal{Z}^{H_p}(\xi\cdot W)\right\}.$
    \end{proof}
\end{thm}

This can be regarded as an \textit{integral analogue} of Waldspurger's multiplicity one result for toric periods in the unramified setting. Indeed, Waldspurger's theorem can be restated as $\mathrm{rank}_\mathbf{C}\ \mathcal{L}^{H_p}_{\pi_p\boxtimes\chi_p}(\mathbf{C})=1.$.
 
  In theory, using our zeta-integral description, one should be able to explicitly show, by direct computation, that the aforementioned linear form take values in $A$ when evaluated on elements of $\mathcal{L}^{H_p}_{\pi_p\boxtimes\chi_p}(A)$. However, this is unsurprisingly very computationally involved, especially in the non-split case. We will actually do this later on in order to obtain explicit formulas for local Shintani functions in terms of Satake parameters and $L$-factor.  
 \section{Shintani functions}\label{secshintani}
 Shintani functions were first introduced by Shintani for symplectic groups. They were then studied in detail in 
 \cite{murase1991whittaker} for $\mathrm{Sp}_{n}$, and later on in \cite{murase1996shintani} for $\GL_{n}$. In this section, mainly motivated by \cite{murase1996shintani}, we define local Shintani-type functions for $(\mathfrak{H},\mathfrak{G})$. Using results from earlier sections, we are able to deduce existence and uniqueness of these functions in a canonical way. Furthermore, we introduce a uniform universal approach, and define a \textit{universal Shintani function} which is a function on $\mathfrak{H}_p\backslash\mathfrak{G}_p/\mathfrak{G}_p^\circ$, but instead takes values in the Hecke algebra $\mathcal{R}^{\mathfrak{G}_p}$. Additionally, we show that the universal Shintani function, actually takes values in $\frac{1}{\#H_Z(\mathbf{F}_p)}\mathcal{R}_{\mathbf{Z}[1/p]}^{\mathfrak{G}_p}$.
 \subsection{Local Shintani functions}\label{seclocalsh}
 Let $\pi_p$ be a (complex) unramified $G_p$-representation of Whittaker type, and $\chi_p$ a (complex) unramified character of $H_p$ with $\omega_{\pi_p}=\chi_p|_{\mathbf{Q}_p^\times}$. To the pair $(\pi_p,\chi_p)$, we attach the \textit{space of local Shintani functions}
 \begin{align}\label{eq:shintani}\mathfrak{Sh}(\pi_p,\chi_p):=\left\{S:\mathfrak{H}_p\backslash\mathfrak{G}_p/\mathfrak{G}_p^\circ\longrightarrow \mathbf{C}\ |\ \mathcal{P}* S=\Theta_{\pi_p\boxtimes \chi_p}(\mathcal{P})S\ \text{for all }\mathcal{P}\in\mathcal{R}^{\mathfrak{G}_p}\right\}\end{align}
 where the $*$-action is as in \Cref{UHM}. 
 It is worth noting that $\mathfrak{H}_p\backslash\mathfrak{G}_p/\mathfrak{G}_p^\circ$ can be canonically identified with $H_p^\circ\backslash G_p/G_p^\circ$. Thus functions in $\mathfrak{Sh}(\pi_p,\chi_p)$ can also be regarded as functions on $H_p^\circ\backslash G_p/G_p^\circ$, which more closely resembles the definitions in \cite{murase1991whittaker} and \cite{murase1996shintani}.
 \begin{thm}\label{thmlocalsh}
Let $\pi_p$ and $\chi_p$ be as above. There is a natural isomorphism 
$$\Hom_{\mathfrak{H}_p}(\pi_p\boxtimes \chi_p,\mathbf{C})\simeq \mathfrak{Sh}(\pi_p,\chi_p)$$
and hence the space of local Shintani functions is one dimensional.
\begin{proof}
    Let $\zeta$ be an element of $\Hom_{\mathfrak{H}_p}(\pi_p\boxtimes \chi_p,\mathbf{C})$. We can associate to it, a function from $\mathfrak{G}_p$ to $\mathbf{C}$, which we call $\mathrm{Sh}_\zeta$. This function is given by $\mathrm{Sh}_\zeta\left((g,h)\right):=\zeta\left((g,h)\cdot (W_{\pi_p}^\mathrm{sph}\otimes 1)\right)$. It follows by construction, that this function factors through $\mathfrak{H}_p\backslash \mathfrak{G}_p/\mathfrak{G}_p^\circ$. Also, a somewhat lengthy, but standard unravelling of definitions, shows that it also satisfies the Hecke eigenvalue property of \eqref{eq:shintani}. Thus,  we obtain an injective linear map
    \begin{align}\label{eq: map 1}\Hom_{\mathfrak{H}_p}(\pi_p\boxtimes \chi_p,\mathbf{C})\hookrightarrow \mathfrak{Sh}(\pi_p,\chi_p)\ ,\ \zeta\mapsto \mathrm{Sh}_\zeta.\end{align}
    Using the proof of \Cref{thmCunivfree} one can deduce every element of $\mathfrak{Sh}(\pi_p,\chi_p)$ is determined by its value on the identity. Hence, \eqref{eq: map 1} is an isomorphism by dimension counting, since we already know that $\Hom_{\mathfrak{H}_p}(\pi_p\boxtimes \chi_p,\mathbf{C})$ is one dimensional.  
\end{proof}
 \end{thm}

 \begin{defn}
     Let $\pi_p$ and $\chi_p$ be as in the beginning of \emph{\Cref{sec unv shintani}}. The normalized local Shintani function attached to the pair $(\pi_p,\chi_p)$, is given by
     $$\mathrm{Sh}_{\pi_p,\chi_p}^{H_p}:\mathfrak{H}_p\backslash \mathfrak{G}_p/\mathfrak{G}_p^\circ\rightarrow\mathbf{C}\ ,\mathrm{Sh}_{\pi_p,\chi_p}^{H_p}(g,h):=\mathcal{Z}^{H_p}((g,h)\cdot (W_{\pi_p}^\mathrm{sph}\otimes 1))$$
     where recall that $\mathcal{Z}^{H_p}$ is the normalized basis element of the one dimensional space $\Hom_{\mathfrak{H}_p}(\pi_p\boxtimes \chi_p,\mathbf{ 1})$, mapping $W_{\pi_p}^\mathrm{sph}\otimes 1$ to $1$.
 \end{defn}

 \subsection{Universal Shintani functions}\label{sec unv shintani} 

 In \Cref{seclocalsh}, to each pair $(\pi_p,\chi_p)$, we have attached a (unique) complex valued Shintani function $\mathrm{Sh}_{\pi_p,\chi_p}^{H_p}$. In this section, we essentially want to combine the data of the functions $\mathrm{Sh}_{\pi_p,\chi_p}^{H_p}$, as we vary $(\pi_p,\chi_p)$, into a single universal Shintani function. This will no longer be a complex valued function, but instead, it will take values in the Hecke algebra $\mathcal{R}^{\mathfrak{G}_p}$. \\
 \\
 We know from \Cref{cor Co and I free}
 that $\mathbf{C}[\mathfrak{G}_p/\mathfrak{G}_p^\circ]_{\mathfrak{H}_p}$ is a free $\mathcal{R}^{\mathfrak{G}_p}$-module of rank one, generated by $[\ch(\mathfrak{G}_p^\circ)]$. This gives a canonical candidate for the universal Shintani function.
 \begin{defn}\label{defunivsh}
     The universal Shintani function $\mathrm{Sh}^{H_p,\mathrm{unv}}$ is the function
     $$\mathrm{Sh}^{H_p,\mathrm{unv}}:\mathfrak{H}_p\backslash \mathfrak{G}_p/\mathfrak{G}_p^\circ\longrightarrow \mathcal{R}^{\mathfrak{G}_p}\ ,\ \mathrm{Sh}^{H_p,\mathrm{unv}}(g,h):=\mathcal{Q}_{(g,h)}^{'}$$
     where $\mathcal{Q}_{(g,h)}$ is the unique Hecke operator in $\mathcal{R}^{\mathfrak{G}_p}$, for which $\mathcal{Q}_{(g,h)}*[\ch(\mathfrak{G}_p^\circ)]=[\ch\left((g,h)\mathfrak{G}_p^\circ)\right)]$ in $\mathbf{C}[\mathfrak{G}_p/\mathfrak{G}_p^\circ]_{\mathfrak{H}_p}$.
 \end{defn}
 \noindent The function $\mathrm{Sh}^{H_p,\mathrm{unv}}$ is well-defined. Before stating and proving the main results of this section, we need the following technical result that closely resembles Proposition $4.4.1$ in \cite{Loeffler_2021}.
\begin{prop}\label{proptech}
    Let $\pi_p$ be an unramified $G_p$-representation and $\chi_p$ an unramified character of $H_p$ with $\omega_{\pi_p}=\chi_p|_{\mathbf{Q}_p^\times}$. Let $\zeta\in\Hom_{\mathfrak{H}_p}(\pi_p\boxtimes \chi_p,\mathbf{C})$, and let $\xi=\sum_i c_i[\ch((g_i,h_i)\mathfrak{G}_p^\circ)]$ in $\mathbf{C}[\mathfrak{G}_p/\mathfrak{G}_p^\circ]_{\mathfrak{H}_p}$. Let $\mathcal{Q}$, be an element of $\mathcal{R}^{\mathfrak{G}_p}$, for which $\xi=\mathcal{Q}*[\ch(\mathfrak{G}_p^\circ)]$. Then
    $$\sum_ic_i\zeta((g_i,h_i)\cdot (W_{\pi_p}^\mathrm{sph}\otimes 1))=\Theta_{\pi_p\boxtimes \chi_p}\left(\mathcal{Q}^{'}\right)\zeta(W_{\pi_p}^\mathrm{sph}\otimes 1).$$
    \begin{proof}
        An unraveling of definitions shows that the map
        $$\mathfrak{z}:\mathbf{C}[\mathfrak{G}_p/\mathfrak{G}_p^\circ]\longrightarrow\mathbf{C}\ ,\ \xi\mapsto \zeta(\xi\cdot (W_{\pi_p}^\mathrm{sph}\otimes 1))$$
        factors through $\mathbf{C}[\mathfrak{G}_p/\mathfrak{G}_p^\circ]_{\mathfrak{H}_p}$. Now, $\mathfrak{z}([\xi])$ is equal to $\mathfrak{z}\left(\mathcal{Q}*[\ch(\mathfrak{G}_p^\circ)]\right)$ which is in turn equal to $\mathfrak{z}([\mathcal{Q}^{'}])$ and hence $\Theta_{\pi_p\boxtimes \chi_p}\left(\mathcal{Q}^{'}\right)\zeta(W_{\pi_p}^\mathrm{sph}\otimes 1)$. The result then follows by also evaluating $\mathfrak{z}(\xi)$ directly using $\xi=\sum_i c_i[\ch((g_i,h_i)\mathfrak{G}_p^\circ)]$. \end{proof}
\end{prop}

\begin{thm}\label{thmunvsh}
             The universal Shintani function satisfies the following properties:
             \begin{enumerate}
                 \item It is the unique function from $\mathfrak{H}_p\backslash \mathfrak{G}_p/\mathfrak{G}_p^\circ$ to $\mathcal{R}^{\mathfrak{G}_p}$ that makes the following diagram
                 \[\begin{tikzcd}
	{\mathfrak{H}_p\backslash \mathfrak{G}_p/\mathfrak{G}_p^\circ} && {\mathcal{R}^{\mathfrak{G}_p}} \\
	&& {\mathbf{C}}
	\arrow["{{\mathrm{Sh}^{H_p}_{\pi_p,\chi_p}}}"', from=1-1, to=2-3]
	\arrow["{{\Theta_{\pi_p\boxtimes \chi_p}}}", from=1-3, to=2-3]
	\arrow["{{\mathrm{Sh}^{H_p,\mathrm{unv}}}}", from=1-1, to=1-3]
\end{tikzcd}\]
commute, for all pairs $(\pi_p,\chi_p)$ where $\pi_p$ is an unramified $G_p$-representation of Whittaker type, and $\chi_p$ is an unramified character of $H_p$ with $\omega_{\pi_p}=\chi_p|_{\mathbf{Q}_p^\times}$.
         \item It is valued in $\frac{1}{\#H_Z(\mathbf{F}_p)}\mathcal{R}_{\mathbf{Z}[1/p]}^{\mathfrak{G}_p}$.
             \end{enumerate}
             \begin{proof}
             The commutativity of the diagram follows from \Cref{proptech}. For uniqueness, suppose that $S:\mathfrak{H}_p\backslash \mathfrak{G}_p/\mathfrak{G}_p^\circ\rightarrow \mathcal{R}^{\mathfrak{G}_p}$ is another such function. Then for any fixed $(g,h)\in\mathfrak{G}_p$, we have
             $$\Theta_{\pi_p\boxtimes \chi_p}\left(\mathrm{Sh}^{H_p,\mathrm{unv}}(g,h)\right)=\Theta_{\pi_p\boxtimes \chi_p}\left(S(g,h)\right)$$
             and this holds for all representations $\pi_p\boxtimes \chi_p$ as in the statement of the result. However, the family of these representations is dense in $\mathrm{Spec}(\mathcal{R}^{\mathfrak{G}_p})$ and the result follows.
             For the third part, it is once again sufficient by \cref{Dns} and \Cref{Ds} to consider the elements $\mathrm{Sh}^{H_p,\mathrm{unv}}\left(\gamma_{\lambda,a}\right)$ if $H$ is non-split, and $\mathrm{Sh}^{H_p,\mathrm{unv}}\left(\gamma_{\lambda,a,b}\right)$ if $H$ is split for $\lambda\in\mathbf{Z}_{\geq 0}$ and $a,b\in\mathbf{Z}$ . These are respectively equal to $\mathcal{Q}_{\gamma_{\lambda,a}}^{'}$ and $\mathcal{Q}_{\gamma_{\lambda,a,b}}^{'}$. By \Cref{lemmap}, these are  equal to $\frac{1}{\#H_Z(\mathbf{F}_p)}p^{\lambda-1}\mathcal{P}_{\xi_{\lambda,a}}^\mathrm{unv}$ and $\frac{1}{\#H_Z(\mathbf{F}_p)}p^{\lambda-1}\mathcal{P}_{\xi_{\lambda,a,b}}^\mathrm{unv}$ respectively. Finally, by \Cref{thmL1}, these are elements of $\frac{1}{\#H_Z(\mathbf{F}_p)}\mathcal{R}_{\mathbf{Z}[1/p]}^{\mathfrak{G}_p}$. This concludes the proof.
             \end{proof}
         \end{thm}
         
        \subsection{Explicit formulas in terms of Satake parameters and $L$-factors}
        We have so far constructed local and universal Shintani functions and have proven their uniqueness. In this section, we will really use the zeta integrals introduced in \Cref{seczeta} in order to obtain explicit formulas for the Shintani functions, in terms of Satake parameters and local $L$-factors attached to each pair $(\pi_p,\chi_p)$, where $\pi_p$ is an unramified $G_p$-representation of Whittaker type and $\chi_p$ is an unramified character of $H_p$ with $\omega_{\pi_p}=\chi_p|_{\mathbf{Q}_p^\times}$. We will denote the Schur polynomial $(x^{n+1}-y^{n+1})/(x-y)$, of homogeneous degree $n$, by $\mathrm{Sch}_n(x,y)$. 
        \subsubsection{Split $H$}\label{secHsplit} We firstly deal with the case where $H$ is split. Since $$\mathrm{Sh}_{\pi_p,\chi_p}^{H_p}(\gamma_{\lambda,a,b})=\chi_p(\left[\begin{smallmatrix}
            p^a & \\
             & 1
        \end{smallmatrix}\right]p^b)\mathrm{Sh}_{\pi_p,\chi_p}^{H_p}(\gamma_{\lambda,0,0})$$it suffices to study in detail the values of $\mathrm{Sh}_{\pi_p,\chi_p}^{H_p}(\gamma_{\lambda,0,0})$ for $\lambda\geq 1$. By construction, this is given by $\mathcal{Z}^{H_p}\left(\gamma_{\lambda,0,0}\cdot (W_{\pi_p}^\mathrm{sph}\otimes 1)\right)$, and by $\mathfrak{H}_p$-equivariance, this in turn is equal to $\mathcal{Z}^{H_p}\left((n_\lambda,\left[\begin{smallmatrix}
            p^\lambda & \\
             & 1
        \end{smallmatrix}\right])\cdot (W_{\pi_p}^\mathrm{sph}\otimes 1)\right)$, where $n_\lambda:=\left[\begin{smallmatrix}
            1 & p^{-\lambda} \\
            & 1
        \end{smallmatrix}\right]$. We evaluate
     \begin{align*}
         \mathcal{Z}^{H_p}\left((n_\lambda,\left[\begin{smallmatrix}
            p^\lambda & \\
             & 1
        \end{smallmatrix}\right])\cdot (W_{\pi_p}^\mathrm{sph}\otimes 1),s+\frac{1}{2}\right)&=\chi_p\left(\left[\begin{smallmatrix}
            p^\lambda & \\
            & 1
        \end{smallmatrix}\right]\right)\int_{\mathbf{Q}_p^\times}(n_\lambda\cdot W_{\pi_p}^\mathrm{sph})(\left[\begin{smallmatrix}
            x & \\
            & 1
        \end{smallmatrix}\right])\chi_p(\left[\begin{smallmatrix}
            x & \\
            & 1
        \end{smallmatrix}\right])^{-1}|x|_p^{s}\ d^\times x\\
        &=\chi_p\left(\left[\begin{smallmatrix}
            p^\lambda & \\
            & 1
        \end{smallmatrix}\right]\right)\int_{\mathbf{Q}_p^\times}\psi(p^{-\lambda} x)W_{\pi_p}^\mathrm{sph}(\left[\begin{smallmatrix}
            x & \\
            & 1
        \end{smallmatrix}\right])\chi_p(\left[\begin{smallmatrix}
            x & \\
             & 1
        \end{smallmatrix}\right])^{-1}|x|_p^s\ d^\times x
     \end{align*}
     Once again using \cite{shintani1976explicit}, one can show that the integral on the right is given by 
     $$L^{H_p}\left(\pi_p\boxtimes \chi_p,s+\frac{1}{2}\right)+\sum_{m=0}^{\lambda-1}\mathrm{Sch}_m(\alpha^{\pi_p},\beta^{\pi_p})\left(\int_{p^m\mathbf{Z}_p^\times}\psi(p^{-\lambda}x)-1\ d^\times x\right)\chi_p(\left[\begin{smallmatrix}
         p & \\
          & 1
     \end{smallmatrix}\right])^{-m}p^{-m(s+\frac{1}{2})}.$$
     An integral computation using orthogonality of characters, shows that for $m=0,\dots,\lambda-1$, one has
     $$\int_{p^m\mathbf{Z}_p^\times}\psi(p^{-\lambda}x)-1\ d^\times x=\begin{cases}
         -\frac{p}{p-1}\ ,\ \text{if }m=\lambda-1\\
         -1\ ,\ \text{otherwise}.
     \end{cases}$$
     We write $\delta_{\chi_p}$ for $\chi_p(\left[\begin{smallmatrix}
         p & \\
         & 1
     \end{smallmatrix}\right])$. It follows that $\mathrm{Sh}_{\pi_p,\chi_p}^{H_p}(\gamma_{\lambda,0,0})$ is given by
        \begin{align}\delta_{\chi_p}^{\lambda}-\frac{p\ \delta_{\chi_p}\mathrm{Sch}_{\lambda-1}(\alpha^{\pi_p},\beta^{\pi_p})}{\#H_Z(\mathbf{F}_p)L^{H_p}\left(\pi_p\boxtimes \chi_p,\frac{1}{2}\right)}p^{-\frac{\lambda-1}{2}}-\sum_{m=0}^{\lambda-2}\delta_{\chi_p}^{\lambda-m}\frac{\mathrm{Sch}_m(\alpha^{\pi_p},\beta^{\pi_p})}{L^{H_p}\left(\pi_p\boxtimes \chi_p,\frac{1}{2}\right)}p^{-\frac{m}{2}}\label{eq:shloc}
        \end{align}
        where the sum on the right is taken to be zero by convention in the case where $\lambda=1$.
        We have thus completely expressed the local Shintani functions in terms of Satake parameters and $L$-factors. This in turn also allows us to write down an explicit formula for the universal Shintani function. Firstly, we note that the Schur polynomial, is invariant under the Weyl-group of $\GL_2$. It has the following alternative description
        \begin{align}\mathrm{Sch}_m(x,y)=\begin{cases}
            (x+y)^m+a_{m,2}(x+y)^{m-2}xy+\dots + a_{m,n-2}(xy)^{\frac{m-2}{2}}\ ,\ \text{ if } m\text{ is even}\\
            (x+y)^m+a_{m,2}(x+y)^{m-2}xy+\dots + a_{m,n-1}(x+y)(xy)^{\frac{m-1}{2}}\ ,\ \text{ if } m\text{ is odd}.
            \end{cases}\label{eq:sch}
            \end{align}
            with $a_{m,i}\in\mathbf{Z}$.
            It is clear that we need to define
            \begin{align}
                \mathrm{Sch}_m^\circ(X,Y):=p^{-\frac{m}{2}}\begin{cases}
                    (p^{-1/2}Y)^m+a_{m,2}(p^{-1/2}Y)^{m-2}X+\dots + a_{m,n-2}X^{\frac{m-2}{2}}\ ,\ \text{ if } m\text{ is even}\\
            (p^{-1/2}Y)^m+a_{m,2}(p^{-1/2}Y)^{m-2}X+\dots + a_{m,n-1}(p^{-1/2}Y)X^{\frac{m-1}{2}}\ ,\ \text{ if } m\text{ is odd}.
                \end{cases}
            \end{align}
            By a standard parity consideration, we see that the polynomial $\mathrm{Sch}_m^\circ(X,Y)$ is an element of $\mathbf{Z}[1/p][X,Y]$.
            It then follows from \Cref{thmunvsh} that $\mathrm{Sh}^{H_p,\mathrm{unv}}\left(\gamma_{\lambda,0,0}\right)^{'}$ is given by
            \begin{align}A_p^\lambda-\left(\frac{pA_p}{\#H_Z(\mathbf{F}_p)}\mathrm{Sch}_{\lambda-1}^\circ(\mathcal{S}_p,\mathcal{T}_p)-\sum_{m=0}^{\lambda-2}A_p^{\lambda-m}\mathrm{Sch}_{m}^\circ(\mathcal{S}_p,\mathcal{T}_p)\right)\mathcal{P}_p^{H}(p^{-1/2})
            \label{eq:shunv}
            \end{align}
            which is indeed an element of $\frac{1}{\#H_Z(\mathbf{F}_p)}\mathcal{R}_{\mathbf{Z}[1/p]}^{\mathfrak{G}_p}$ as shown in \Cref{thmunvsh} part $(3)$
            \subsubsection{Non-split $H$}\label{secHnonsplit} We now treat the case where $H$ is non-split. By the same considerations as in \Cref{secHsplit}, it is enough to obtain a formula for the value of $\mathrm{Sh}_{\pi_p,\chi_p}^{H_p}\left(s(\lambda),1\right)$ for $\lambda\in\mathbf{Z}_{\geq 0}$. A standard matrix computation using the first matrix identity in the proof of \Cref{lemlattices} shows that the subgroup $H_p^\circ(\lambda)$ stabilizes $s(\lambda)\cdot W_{\pi_p}^\mathrm{sph}$. For $\gamma\in H_p^\circ$, we write $\gamma s(\lambda)=p^{z_\gamma}\left[\begin{smallmatrix}
        p^{r_\gamma} & \\
        & 1
    \end{smallmatrix}\right]n_\gamma k_\gamma\in Z_pA_pN_p G_p^\circ$. Once again we idenitify $N_p$ with $\mathbf{Q}_p$ so we can talk about the $p$-adic valuation of an element of $N_p$. It follows from \Cref{defzetaavg} and the computation in \Cref{secHsplit} that
    \begin{align*}
       \mathcal{Z}_\mathrm{avg}^H\left((s(\lambda),1)\cdot (W_{\pi_p}^\mathrm{sph}\otimes 1),s+\frac{1}{2}\right)&=\int_{H_p^\circ}Z^{H_p}\left(h\cdot (s(\lambda)\cdot W_{\pi_p}^\mathrm{sph})\otimes \chi_p,s\right)\ dh\\
       &=\vol_{H_p}(H_p^\circ(\lambda))\sum_{\gamma\in H_p^\circ/H_p^\circ(\lambda)}\int_{\mathbf{Q}_p^\times}(\gamma s(\lambda)\cdot W_{\pi_p}^\mathrm{sph})(\left[\begin{smallmatrix}
           x & \\
           & 1
       \end{smallmatrix}\right])|x|_p^{s}\ d^\times x\\
       &=\scalemath{0.9}{\vol_{H_p}(H_p^\circ(\lambda))\sum_{\gamma\in H_p^\circ/H_p^\circ(\lambda)}(\alpha^{\pi_p}\beta^{\pi_p})^{z_\gamma} p^{r_\gamma s}\int_{\mathbf{Q}_p^\times} \psi(p^{v_p(n_\gamma)}x) W_{\pi_p}^\mathrm{sph}(\left[\begin{smallmatrix}
           x & \\
            & 1
       \end{smallmatrix}\right])|x|_p^s\ d^\times x.}
    \end{align*}
    It then follows that $\mathcal{Z}^{H_p}((s(\lambda),1)\cdot (W_{\pi_p}^\mathrm{sph}\otimes 1))$ is given by
    \begin{align}
        \vol_{H_p}(H_p^\circ(\lambda))\sum_{\gamma\in H_p^\circ/H_p^\circ(\lambda)}\left(\alpha^{\pi_p}\beta^{\pi_p}\right)^{z_\gamma}\left(1+\sum_{m=0}^{-v_p(n_\gamma)-1}\epsilon_{m,v_p(n_\gamma)}\frac{\mathrm{Sch}_m(\alpha^{\pi_p},\beta^{\pi_p})}{L^{H_p}(\pi_p\boxtimes \chi_p,\frac{1}{2})}p^{-\frac{m}{2}}\right)\label{eq:10}
    \end{align}
    where the sum is understood to be zero whenever $v_p(n_\gamma)\geq 0$, and $\epsilon_{m,i}:=\int_{p^m\mathbf{Z}_p^\times}\psi(p^{i} x)-1 d^\times x$.
    We now need the following technical lemma.
    \begin{lem}\label{lemtech}
       There exists a complete set of distinct coset representatives of $H_p^\circ/H_p^\circ(\lambda)$ that can be partitioned as follows:
       \begin{itemize}
          \item There are $\varphi(p^\lambda)$ elements for which $v_p(n_\gamma)=-\lambda$ and $z_\gamma=0$\\
          \item For each $1\leq i\leq \lambda-1$, there  are $\varphi(p^{\lambda-i})$ elements for which $v_p(n_\gamma)=i-\lambda$ and $z_\gamma= i$. Additionally, there are $\varphi(p^{\lambda-i})$ elements for which $v_p(n_\gamma)=i-\lambda$ and $z_\gamma= 0$.\\
          \item Finally, there's a unique element for which $n_\gamma=0$ and $z_\gamma=\lambda$, and there is a unique element for which $n_\gamma=0$ and $z_\gamma=0$.
       \end{itemize}
        \begin{proof}
            See Appendix A.
        \end{proof}
    \end{lem}

\noindent Applying \Cref{lemtech}, we see that \eqref{eq:10} divided by $\vol_{H_p}(H_p^\circ(\lambda))$, is given by 
\begin{align}
    \scalemath{0.9}{\varphi(p^\lambda)\left(1+ \sum_{m=0}^{\lambda-1}\epsilon_{m,-\lambda}\frac{\mathrm{Sch}_m(\alpha^{\pi_p},\beta^{\pi_p})}{L^{H_p}(\pi_p\boxtimes \chi_p,\frac{1}{2})}p^{-\frac{m}{2}}\right)+\sum_{i=1}^\lambda\varphi(p^{\lambda-i})\left(1+(\alpha^{\pi_p}\beta^{\pi_p})^i\right)\left(1+ \sum_{m=0}^{\lambda-i-1}\epsilon_{m,i-\lambda}\frac{\mathrm{Sch}_m(\alpha^{\pi_p},\beta^{\pi_p})}{L^{H_p}(\pi_p\boxtimes \chi_p,\frac{1}{2})}p^{-\frac{m}{2}}\right).} \label{eq:11}
\end{align}

\noindent Once again, for $i=\lambda$ the sum over $m$ is understood to be zero. For $-1\leq m\leq \lambda-2$, we define the polynomial
$$\mu_{m,\lambda}(T):=\varphi(p)\left(1-\mathrm{Sch}_{\lambda-m-2}(1,p^{-1})-\mathrm{Sch}_{\lambda-m-2}(1,p^{-1}T)\right)-p^{2+m-\lambda}\left(1+T^{\lambda-m-1}\right)\in\mathbf{Z}[1/p][T].$$
A lengthy but standard algebraic manipulation of \eqref{eq:11}, gives that $\#H_Z(\mathbf{F}_p)\mathrm{Sh}_{\pi_p,\chi_p}^{H_p}\left(s(\lambda),1\right)$ is given by 
\begin{align}
    -\mu_{-1,\lambda}(\alpha^{\pi_p}\beta^{\pi_p})-\frac{p\ \mathrm{Sch}_{\lambda-1}(\alpha^{\pi_p},\beta^{\pi_p})}{L^{H_p}(\pi_p\boxtimes \chi_p,\frac{1}{2})}p^{-\frac{\lambda-1}{2}}+\sum_{m=0}^{\lambda-2}\mu_{m,\lambda}(\alpha^{\pi_p}\beta^{\pi_p})\frac{\mathrm{Sch}_m(\alpha^{\pi_p},\beta^{\pi_p})}{L^{H_p}(\pi_p\boxtimes \chi_p,\frac{1}{2})}p^{-\frac{m}{2}}.
\end{align}
where once again, the sum on the right is understood to be zero if $\lambda=1$. From this, and \Cref{thmunvsh} we can also deduce that $p^{\lambda-1}\#H_Z(\mathbf{F}_p)\mathrm{Sh}^{H_p,\mathrm{unv}}\left(s(\lambda),1\right)^{'}$ is given explicitly by 
\begin{align}\label{eq: 16}
    -\mu_{-1,\lambda}(\mathcal{S}_p)-\left(p\ \mathrm{Sch}_{\lambda-1}^\circ(\mathcal{S}_p,\mathcal{T}_p) +\sum_{m=0}^{\lambda-2}\mu_{m,\lambda}(\mathcal{S}_p)\mathrm{Sch}_m^\circ(\mathcal{S}_p,\mathcal{T}_p   )     \right)\mathcal{P}_p^{H}(p^{-1/2})
\end{align}
and thus $\mathrm{Sh}^{H_p,\mathrm{unv}}(s(\lambda),1)$ is indeed an element of $\frac{1}{\#H_Z(\mathbf{F}_p)}\mathcal{R}_{\mathbf{Z}[1/p]}^{\mathfrak{G}_p}$, as shown in part $(3)$ of \Cref{thmunvsh}.\\
\\
\noindent Let $n_H$ be the identity matrix if $H$ is non-split, and the matrix $n_0$ if $H$ is split. The results of \Cref{secHsplit} and \Cref{secHnonsplit} result in the following.
\begin{thm}\label{thm intro expl loc sh}
    The local Shintani function attached to a pair $(\pi_p,\chi_p)$, is determined by the value of $\mathrm{Sh}_{\pi_p,\chi_p}^{H_p}(n_Hs(\lambda),1)$, which is explicitly given by the formula 
    $$\scalemath{0.95}{\begin{dcases*}
        \delta_{\chi_p}^{\lambda}-\frac{p\ \delta_{\chi_p}\mathrm{Sch}_{\lambda-1}(\alpha^{\pi_p},\beta^{\pi_p})}{\#H_Z(\mathbf{F}_p)L^{H_p}\left(\pi_p\boxtimes \chi_p,\frac{1}{2}\right)}p^{-\frac{\lambda-1}{2}}-\sum_{m=0}^{\lambda-2}\delta_{\chi_p}^{\lambda-m}\frac{\mathrm{Sch}_m(\alpha^{\pi_p},\beta^{\pi_p})}{L^{H_p}\left(\pi_p\boxtimes \chi_p,\frac{1}{2}\right)}p^{-\frac{m}{2}}\ ,\ H\text{ split}\\
        -\frac{\mu_{-1,\lambda}(\alpha^{\pi_p}\beta^{\pi_p})}{\#H_Z(\mathbf{F}_p)}-\frac{p\ \mathrm{Sch}_{\lambda-1}(\alpha^{\pi_p},\beta^{\pi_p})}{\#H_Z(\mathbf{F}_p)L^{H_p}(\pi_p\boxtimes \chi_p,\frac{1}{2})}p^{-\frac{\lambda-1}{2}}+\sum_{m=0}^{\lambda-2}\frac{\mu_{m,\lambda}(\alpha^{\pi_p}\beta^{\pi_p})}{\#H_Z(\mathbf{F}_p)}\frac{\mathrm{Sch}_m(\alpha^{\pi_p},\beta^{\pi_p})}{L^{H_p}(\pi_p\boxtimes \chi_p,\frac{1}{2})}p^{-\frac{m}{2}}\ ,\ H\text{ non-split.}
    \end{dcases*}}$$
\end{thm}
\begin{cor}\label{corl intro expl unv sh}
    The universal Shintani function, is determined by $\mathrm{Sh}^{H_p,\mathrm{unv}}\left(n_Hs(\lambda),1\right)^{'}$, which is explicitly given by the formula
  
$$\begin{dcases*}
            A_p^\lambda-\left(\frac{p\ \mathrm{Sch}_{\lambda-1}^\circ(\mathcal{S}_p,\mathcal{T}_p)}{\#H_Z(\mathbf{F}_p)}A_p-\sum_{m=0}^{\lambda-2}A_p^{\lambda-m}\mathrm{Sch}_{m}^\circ(\mathcal{S}_p,\mathcal{T}_p)\right)\mathcal{P}_p^{H}(p^{-1/2})\ ,\ H\text{ split}\\
            -\frac{\mu_{-1,\lambda}(\mathcal{S}_p)}{\#H_Z(\mathbf{F}_p)}-\left(\frac{p\ \mathrm{Sch}_{\lambda-1}^\circ(\mathcal{S}_p,\mathcal{T}_p)}{\#H_Z(\mathbf{F}_p)}\  +\sum_{m=0}^{\lambda-2}\frac{\mu_{m,\lambda}(\mathcal{S}_p)}{\#H_Z(\mathbf{F}_p)}\mathrm{Sch}_m^\circ(\mathcal{S}_p,\mathcal{T}_p   )     \right)\mathcal{P}_p^{H}(p^{-1/2})\ ,\ H\text{ non-split}.
        \end{dcases*}
        $$
\end{cor}

\begin{rem}
    Plugging $\lambda=1$ in \Cref{corl intro expl unv sh}, one can quickly calculate that
    \begin{align}\label{eq:optimal}\mathrm{Sh}^{H_p,\mathrm{unv}}(n_Hs(1),1)^{'}=\begin{dcases*}
        A_p -\frac{pA_p}{\#H_Z(\mathbf{F}_p)}\mathcal{P}_p^{H}(p^{-1/2})\ ,\ H\text{ split}\\
        \frac{1}{\#H_Z(\mathbf{F}_p)}\mathcal{T}_p\ ,\ H\text{ non-split}.
    \end{dcases*}\end{align}
    Thus part $(2)$ of \Cref{thmunvsh} is indeed optimal in the sense that there is no cancellation of the $\frac{1}{\#H_Z(\mathbf{F}_p)}$ factor.
    \end{rem}

    \section{Optimality of integral abstract norm-relations}\label{sec L_1^H}
    The goal of this section is to finish off the proof of \Cref{thmintro1}. It remains to show that the universal Hecke operators $\mathcal{P}_\xi^\mathrm{unv}$ attached to integral test data $\xi$ in the lattice $\mathcal{L}_1^{H_p}$ lie in the ideal $$\mathfrak{h}^{H_p}=\begin{dcases*}
\left\langle\#H_Z(\mathbf{F}_p),\mathcal{T}_p\right\rangle\ ,\  H\text{ non-split}\\
\left\langle\#H_Z(\mathbf{F}_p),{\mathcal{P}_p^{H}}'(p^{-1/2})\right\rangle\ ,\  H\ \text{split}
\end{dcases*}$$
of $\mathcal{R}_{\mathbf{Z}[1/p]}^{\mathfrak{G}_p}$.
\noindent The following technical lemma will be used for the non-split case. We provide a proof of this in the appendix.
\begin{lem}\label{lemideal}
    For any $\lambda\in\mathbf{Z}_{\geq 1}$, the characteristic function $\ch(G_p^\circ s(\lambda) G_p^\circ)$ is an element of the ideal 
    $\left(p+1,\mathcal{T}_p\right)\subseteq\mathcal{H}_{G_p}^\circ(\mathbf{Z}[1/p]).$
    \begin{proof}
       See \Cref{appendix lemma ideal}.
    \end{proof}
\end{lem}
\begin{thm}\label{thm L1 7}
   Let $\xi_0$ be the characteristic function of $\mathfrak{H}_p\mathfrak{G}_p^\circ$. The following are true:
 \begin{enumerate}
 \item For $\xi\in$ any integral vector of level $\mathfrak{G}(\mathbf{Z}_p)$, we have 
 $$\mathfrak{z}(\xi)=\mathcal{P}_\xi^\mathrm{unv}\cdot \mathfrak{z}(\xi_0)\ \ \ \mathrm{with}\ \ \mathcal{P}_\xi^\mathrm{unv}\in \mathcal{R}_{\mathbf{Z}[1/p]}^{\mathfrak{G}_p}.$$
 \item
      For $\xi_1\in C_c^\infty(\mathfrak{H}(\mathbf{Q}_p)\backslash \mathfrak{G}(\mathbf{Q}_p)/\mathfrak{G}(\mathbf{Z}_p)[p],\mathbf{Z}[1/p])$ any integral vector of level $\mathfrak{G}(\mathbf{Z}_p)[p]$, we have
     $$\mathrm{norm}^{\mathfrak{G}(\mathbf{Z}_p)[p]}_{\mathfrak{G}(\mathbf{Z}_p)}\mathfrak{z}(\xi_1)=\mathcal{P}_{\mathrm{Tr}(\xi_1)}^\mathrm{unv}\cdot\mathfrak{z}(\xi_0)\ \ \mathrm{with}\ \  
     \mathcal{P}_{\mathrm{Tr}(\xi_1)}^\mathrm{unv}\in\mathfrak{h}^{H_p}.$$
    \end{enumerate}
\begin{proof}
    By \Cref{thmL1} it only remains to prove that $\mathcal{P}_{\xi}^{\mathrm{unv}}\in\mathfrak{h}^{H_p}$ for any $\xi\in\mathcal{L}_1^{H_p}.$ By freeness, it suffices in both cases to prove this on characteristic functions. If $H$ is non-split (resp. split) and
    $\xi\in\mathcal{L}_1^{H_p}$ is of the form $\#H_Z(\mathbf{F}_p)\xi_{0,a}$ (resp. $\#H_Z(\mathbf{F}_p)\xi_{0,a,b}$) then there is nothing to prove, since the universal Hecke operator in that case is simply given by $\#H_Z(\mathbf{F}_p)\mathcal{S}_p^a$ (resp. $\#H_Z(\mathbf{F}_p)A_p^a\mathcal{S}_p^b$), which are already elements of the ideals in question. Thus, we may suppose we are working with a characteristic function of the form $\xi_{\lambda,a}$ (resp. $\xi_{\lambda,a,b}$) for $\lambda\geq 1$. We firstly deal with the case where $H$ is non-split. We know from \Cref{thmCunivfree} that $\xi_{\lambda,a}$ is given by $A_p^{-a}\phi_{s(\lambda)^{-1}}* \xi_{0,0}$. Thus the result follows at once from \Cref{lemideal}. 
    Thus, from now on, we assume that $H$ is split. As we've seen in the proof of \Cref{thmCunivfree}, if $H$ is split, then we do not have the same luxury, and it is combinatorially extremely hard to write down $\mathcal{P}_{\xi_{\lambda,a,b}}^{H,\mathrm{unv}}$ directly using the recurrence relation of \Cref{thmCunivfree}. Instead, we can argue as follows. By definition, we have $\xi_{\lambda,a,b}=\mathcal{P}_{\xi_{\lambda,a,b}}^{H,\mathrm{unv}}*\xi_{0,0,0}$. Applying the Hecke equivariant map $i$, and using \Cref{lemmap}, this is equal to 
    \begin{align*}\mathcal{P}_{\xi_{\lambda,a,b}}^{H,\mathrm{unv}}*[\ch(\mathfrak{G}_p^\circ)]&=\#H_Z(\mathbf{F}_p)p^{\lambda-1}A_p^{-a}\mathcal{S}_p^{-b}*[\ch((n_0s(\lambda),1)\mathfrak{G}_p^\circ)]\\
    &=\#H_Z(\mathbf{F}_p)p^{\lambda-1}A_p^{-a}\mathcal{S}_p^{-b}\mathrm{Sh}^{H_p,\mathrm{unv}}(n_0s(\lambda),1)^{'}*[\ch(\mathfrak{G}_p^\circ)]
    \end{align*}
    in $\mathbf{C}[\mathfrak{G}_p/\mathfrak{G}_p^\circ]_{\mathfrak{H}_p}$ . Thus, by \Cref{cor Co and I free} 
    $$\frac{A_p^{a} \mathcal{S}_p^{b}}{p^{\lambda-1}}\mathcal{P}_{\xi_{\lambda,a,b}}^{H,\mathrm{unv}^{'}}=\#H_Z(\mathbf{F}_p)\mathrm{Sh}^{H_p,\mathrm{unv}}(n_0s(\lambda),1)^{'}\in\mathcal{R}_{\mathbf{Z}[1/p]}^{\mathfrak{G}_p}.$$
    By \Cref{thmunvsh} for the split case, the right hand side of the above equality is equal to $$\#H_Z(\mathbf{F}_p)A_p^\lambda-\left(p\ \mathrm{Sch}_{\lambda-1}^\circ(\mathcal{S}_p,\mathcal{T}_p)A_p-\#H_Z(\mathbf{F}_p)\sum_{m=0}^{\lambda-2}A_p^{\lambda-m}\mathrm{Sch}_m^\circ(\mathcal{S}_p,\mathcal{T}_p)\right)\mathcal{P}_p^{H}(p^{-1/2}).$$
    This shows that every such Hecke operator is contained in the ideal in question. The fact that all generators are attainable follows from \eqref{eq:optimal}.
\end{proof}
\end{thm}
\begin{rem}\label{rem non-split ideal h}
    We note that in the non-split case the ideal $\left\langle\#H_Z(\mathbf{F}_p),\mathcal{T}_p\right\rangle$ is strictly contained in the ideal $I:=\langle\#H_Z(\mathbf{F}_p),p+\mathcal{S}_p',{\mathcal{P}_p^{H}}'(p^{-1/2})\rangle$ of $\mathcal{R}_{\mathbf{Z}[1/p]}^{\mathfrak{G}_p}$, where $\mathcal{P}_p^{H}(X)$ is as in \Cref{lem euler factor}. Of course, for trivial central character these two coincide with the ideal $\langle\#H_Z(\mathbf{F}_p),{\mathcal{P}_p^{H}}'(p^{-1/2})\rangle$ in the corresponding quotient of the Hecke algebra, which is of the same shape as the split case. But for arbitrary central character the larger ideal $I$ is not optimal, whereas $\mathfrak{h}^{H_p}$ is.
\end{rem}
\section{Toric periods for modular forms}\label{sec global}
We fix once and for all a prime $\ell$ and an isomorphism $\iota:\mathbf{C}\simeq\overline{\mathbf{Q}}_\ell$. Let $E=\mathbf{Q}(\sqrt{-M})$ be a quadratic field and $\mathscr{H}$ be the globally non-split torus over $\mathbf{Q}$ given by $\mathrm{Res}_{E/\mathbf{Q}}\GL_1$. We regard $\mathscr{H}$ as a subgroup of $\GL_2$ by $$a+b\sqrt{-M}\mapsto\left[\begin{smallmatrix}
    a & b\\
    -bM & a
\end{smallmatrix}\right].$$
Then for each prime $p\nmid M$, $\mathscr{H}$ is an unramified maximal torus of $\GL_2$ over $\mathbf{Q}_p$ as in \Cref{GHM}. We set $\mathscr{G}:=\GL_2\times \mathscr{H}$ and we regard $\mathscr{H}$ as a subgroup of $\mathscr{G}$ under the usual embedding $h\mapsto(h,h^{-1})$.
\subsection{Automorphic modular forms}\label{sec mod forms.}
     L $f\in S_k(\Gamma_1(N),\varepsilon)$ be a normalized cuspidal eigenform of even integral weight and write $\pi_f=\otimes_p^{'}\pi_{f,p}$ for the irreducible cuspidal automorphic representation of $\GL_2$ attached to $f$ (\cite{gelbart2006automorphic}). Let $S$ be a finite set of places containing $\{\infty,p|2\ell NM\}$ and $\chi_f^S$ be an unramified character of $\mathscr{H}(\mathbf{A}^S)$ taking values in a number field, and whose central character agrees with the $S$-finite part of the adelization of $\varepsilon$. By Waldspurger's multiplicity one theorem for toric periods (\Cref{mult1}), and the same argument as in \cite{harris2001note}, there exists a unique non-zero period
$$\mathcal{Z}_f\in\mathrm{Hom}_{\mathscr{H}(\mathbf{A}^S)}(\pi_f^S\boxtimes\chi_f^S,\mathbf{1})$$
normalized to send the canonical spherical vector $W_{\pi_f^S\boxtimes\chi_f^S}^\mathrm{sph}$ to $1$. Let $L_f$ be the composite of the number field of $f$ \cite[\S $6.5$]{diamond2005first} and the number field of $\chi_f^S$. Let $\mathbf{L}_f/\mathbf{Q}_\ell$ be the smallest $\ell$-adic number field containing the image of $L_f$ under $\iota.$

     Recall the $\mathcal{O}_{\mathbf{L}_f}$-module $\mathcal{L}_{\pi_{f,p}\boxtimes \chi_{f,p}}^{H_p}(\mathcal{O}_{\mathbf{L}_f})$ defined in \Cref{sec int structures} for each prime $p\notin S$.
     \begin{thm}\label{thm appl mod forms}
        Let $p\notin S$ be a prime. Then,
         $$\mathrm{rank}_{\mathcal{O}_{\mathbf{L}_f}}\ \mathcal{L}_{\pi_{f,p}\boxtimes \chi_{f,p}}^{H_p}(\mathcal{O}_{\mathbf{L}_f})=1$$
         \begin{proof}
             As in \cite{loeffler2012computation}, each such local factor $\pi_{f,p}$ is given by the irreducible unramified principal series $\mathcal{B}(\sigma_{p,1},\sigma_{p,2})$ with $\sigma_{p,1}(p)=\alpha_pp^{\frac{1-k}{2}}$ and $\sigma_{p,2}(p)=\beta_pp^{\frac{1-k}{2}}$ where $\alpha_p$ and $\beta_p$ are the roots of the Hecke polynomial at $p$:
        $$X^2-a_p(f)X+\varepsilon(p)p^{k-1}.$$
        Hence by construction, and the fact that $p\in\mathcal{O}_{\mathbf{L}_f}^\times$, $a_p(f)\in \mathcal{O}_{\mathbf{L}_f}$ and $k$ is even, we see that the spherical Hecke eigensystem of $\pi_{f,p}\boxtimes \chi_{f,p}$ restricts to a morphism 
        $\mathcal{R}_{\mathcal{O}_{\mathbf{L}_f}}^{H_p}\rightarrow\mathcal{O}_{\mathbf{L}_f}.$ The result now follows from \Cref{integral mult one}.
         \end{proof}
     \end{thm}

    For a finite set of primes $S^{'}$ disjoint from $S$, we write $$\mathscr{G}(\hat{\mathbf{Z}}^S)[S^{'}]:=\left\{\prod_{p\in S^{'}} \mathscr{G}(\mathbf{Z}_p)[p]\right\}\times\left\{\prod_{p\notin S\cup S^{'}}\mathscr{G}(\mathbf{Z}_p)\right\}$$ 
    where the open compact level subgroup $\mathscr{G}(\mathbf{Z}_p)[p]$ is that of \Cref{def level subgroup}. Finally, we fix the Haar measure $\vol_{\mathscr{H}(\mathbf{A}^S)}:=\prod_{p\not\in S}\vol_{\mathscr{H}(\mathbf{Q}_p)}$ where each $\vol_{\mathscr{H}(\mathbf{Q}_p)}$ is the normalized Haar measure on $\mathscr{H}(\mathbf{Q}_p)$ giving $\mathscr{H}(\mathbf{Z}_p)$ volune $1$.
\begin{thm}\label{thm integr of toric periods} Let $f\in S_k(\Gamma_1(N),\varepsilon)$ be a normalized cuspidal eigenform of even integral weight. Then the following are true:
   \begin{enumerate}\item For any $g\in\mathscr{G}(\mathbf{A}^S)$ and any $C_g\in \vol_{\mathscr{H}(\mathbf{A}^S)}(\mathscr{H}(\mathbf{A}^S)\cap g\mathscr{G}(\hat{\mathbf{Z}}^S)g^{-1})^{-1}\cdot \mathcal{O}_{\mathbf{L}_f}$, the linear form $\mathcal{Z}_f$ satisfies 
   $$C_g\mathcal{Z}(g\cdot W_{\pi_f^S\boxtimes\chi_f^S}^\mathrm{sph})\in\mathcal{O}_{\mathbf{L}_f}.$$
    \item If moreover $C_g\in \mathrm{vol}_{\mathscr{H}(\mathbf{A}^S)}(\mathscr{H}(\mathbf{A}^S)\cap g \mathscr{G}(\hat{\mathbf{Z}}^S)[S']g^{-1})\cdot\mathcal{O}_{L_f}$ for a finite set of primes $S'$ disjoint from $S$, then the linear form $\mathcal{Z}_f$ satisfies
    $$C_g\mathcal{Z}_f(g\cdot W_{\pi_f^S\boxtimes\chi_f^S}^\mathrm{sph})\in\left\{\prod_{\substack{p\in S'\\ \mathrm{inert}}}\left\langle p+1, a_p(f)\right\rangle\right\}\times\left\{ \prod_{\substack{p\in S'\\ \mathrm{split}}}\left\langle p-1, L_p^{-1}\left(f;\chi_{f,p}(\left[\begin{smallmatrix}
        1/p & \\
        & 1
    \end{smallmatrix}\right])\right) \right\rangle\right\}\subseteq \mathcal{O}_{\mathbf{L}_f}$$
    where $a_p(f)$ is the $p$-th Fourier coefficient of $f$, and $L_p^{-1}(f;X):=1-a_p(f)p^{-\tfrac{k}{2}}X+\varepsilon(p)p^{-1}X^2\in\mathcal{O}_{\mathbf{L}_f}[X].$
    \end{enumerate}
    \begin{proof}
        The period $\mathcal{Z}_f$ can be realized as a product of local periods $\prod_{p\notin S}\mathcal{Z}^{\mathscr{H}(\mathbf{Q}_p)}$ where at each prime the period $\mathcal{Z}^{\mathscr{H}(\mathbf{Q}_p)}$ is that of \Cref{def linear form} depending on whether $\mathscr{H}$ is split or non-split over $\mathbf{Q}_p$. The first part then follows from the description of the local principal-series given in \Cref{thm appl mod forms}, and \Cref{integral mult one}. For the second part, it suffices to prove it for $S'=\{p\}$ is a single prime. Using the first part, it is enough to show that \begin{align}\label{eq: expression}\vol_{\mathscr{H}(\mathbf{Q}_p)}(\mathscr{H}(\mathbf{Q}_p)\cap g_p\mathscr{G}(\mathbf{Z}_p)[p]g_p^{-1})^{-1}\cdot\mathcal{Z}^{\mathscr{H}(\mathbf{Q}_p)}(g_p\cdot W_{\pi_{f,p}\boxtimes\chi_{f,p}}^\mathrm{sph})\end{align}
        lies in the $p$-part of the ideal present in the second part of the theorem. By \Cref{proptech}, expression \eqref{eq: expression} is equal to $\Theta_{\pi_{f,p}\boxtimes\chi_{f,p}}(\mathcal{P}_{\xi_p}')$ where 
        $$\xi_p:=\vol_{\mathscr{H}(\mathbf{Q}_p)}(\mathscr{H}(\mathbf{Q}_p)\cap g_p\mathscr{G}(\mathbf{Z}_p)[p]g_p^{-1})^{-1}\cdot \ch(g_p \mathscr{G}(\mathbf{Z}_p)).$$
        But by \Cref{lemlattices} such an element is the image of an element in $\mathcal{L}^{\mathscr{H}(\mathbf{Q}_p)}_1$. Thus by \Cref{thm L1 7}, $\mathcal{P}_{\xi_p}\in\mathfrak{h}^{\mathscr{H}(\mathbf{Q}_p)}$ where the latter is the ideal given in \Cref{sec L_1^H}. The result then follows at once using well-known expressions for spherical Hecke eigenvalues.
    \end{proof}
\end{thm}

\appendix
\section{Technical lemmas}
This appendix includes proofs of \Cref{Dns}, \Cref{Ds}, \Cref{lemtech} and \Cref{lemideal}. We start by proving the first two decomposition results. The following standard result will be used in the proof of \Cref{Dns}. We state it here for completeness.
\begin{prop}\label{propserre}
    Let $p$ be an odd prime. Write $x=up^n\in\mathbf{Q}_p^\times$. Then $x$ is a square if and only if $n$ is even and $u$ is a square in $\mathbf{Z}_p^\times/1+p\mathbf{Z}_p$
    \begin{proof}
        \cite{serre2012course} Section $3.3$.
    \end{proof}
\end{prop}
\begin{lem}\label{techdecomp 1}
    For $H$ non-split, we have the following double coset decomposition $$G_p=\bigcup_{\lambda\in\mathbf{Z}_{\geq 0}}H_ps(\lambda) G_p^\circ$$where $s(\lambda):=\left[\begin{smallmatrix}
        p^\lambda&\\
        &1
    \end{smallmatrix}\right]$. Furthermore, the decomposition is disjoint.
    \begin{proof}
        By Iwasawa decomposition for $\GL_2(\mathbf{Q}_p)$ we have that \begin{align}\label{eq:app 1}G_p=\bigcup_{\substack{r_1,r_2,r_3\in\mathbf{Z}_{\geq 0}\\
        u\in\mathbf{Z}_p^\times\cup \{0\}}} H_p\left[\begin{smallmatrix}
            p^{r_1}& up^{r_2}\\
            & p^{r_3}
        \end{smallmatrix}\right]G_p^\circ\end{align}For $u=0$ note that $$\left[\begin{smallmatrix}
            & 1\\
            -D
        \end{smallmatrix}\right]\left[\begin{smallmatrix}
            p^{r_1}&\\
            &p^{r_3}
        \end{smallmatrix}\right]\left[\begin{smallmatrix}
            & -D^{-1}\\
            1&
        \end{smallmatrix}\right]=\left[\begin{smallmatrix}
            p^{r_3} & \\
            & p^{r_1}
        \end{smallmatrix}\right]$$
        and $\left[\begin{smallmatrix}
            & -D^{-1}\\
            1&
        \end{smallmatrix}\right]$ is in $G_p^\circ$. Hence, we may assume that $r_3\leq r_1$ and in this case we are done. We can now assume that $u$ is a unit. It suffices to show that each double coset in \eqref{eq:app 1} is generated by a diagonal matrix.\\
        \\
        If $r_2\geq r_1$ we have $\left[\begin{smallmatrix}
            p^{r_1}& up^{r_2}\\
            & p^{r_3}
        \end{smallmatrix}\right]=\left[\begin{smallmatrix}
            p^{r_1}&\\
            &p^{r_2}
        \end{smallmatrix}\right]\left[\begin{smallmatrix}
            1& up^{r_2-r_1}\\
            &1
        \end{smallmatrix}\right]$ and the right most matrix is in $G_p^\circ$ since $r_2-r_1\geq 0$, so we are done. As a result, from now, we also assume that $r_1>r_2$.\\
        \\
        \textbf{Case 1:}$(r_2>r_3)$. In this case we use the following matrix relation $$\left[\begin{smallmatrix}
            Dp^{r_2-r_3}& u^{-1}\\
            -u^{-1}D & Dp^{r_2-r_2}
        \end{smallmatrix}\right]\left[\begin{smallmatrix}
            p^{r_1}& up^{r_2}\\
            & p^{r_3}
        \end{smallmatrix}\right]=\left[\begin{smallmatrix}
            p^{r_3}&\\
            &p^{r_1}
        \end{smallmatrix}\right]\left[\begin{smallmatrix}
            Dp^{r_1+r_2-2r_3}&u^{-1}+Dup^{2(r_2-r_3)}\\
            -Du^{-1} &
        \end{smallmatrix}\right]$$where the left most matrix is an element of $H_p$ and the right most matrix is an element of $G_p^\circ$ since we are assuming that $r_1>r_2>r_3$. So in this case we are once again done.\\
        \\
        \textbf{Case 2:}$(r_2\leq r_3$). In this case we use the following matrix relation instead $$\left[\begin{smallmatrix}
            p^{r_3-r_2}&-u\\
            uD&p^{r_3-r_2}
        \end{smallmatrix}\right]\left[\begin{smallmatrix}
            p^{r_1} & up^{r_2}\\
            & p^{r_3}
        \end{smallmatrix}\right]=\left[\begin{smallmatrix}
            p^{r_1+r_3-r_2}&\\
            & p^{r_2}
        \end{smallmatrix}\right]\left[\begin{smallmatrix}
            1& \\
            uDp^{r_1-r_2} & u^2D+p^{2(r_3-r_2)}
        \end{smallmatrix}\right]$$where again the left most matrix is in $H_p$. Now, the right most matrix is clearly in $G_p^\circ$ when $r_2<r_3$. We need tocheck it's again an element of $G_p^\circ$ even when $r_2=r_3$. That is, we need to check that $u^2D+1$ is a unit in $\mathbf{Z}_p$, but this follows from \Cref{propserre}. Thus, our claim holds and we are done in all cases. \\
        \\
        Now we check disjointedness. Suppose we have an equality of double cosets 
        $$H_pt(\lambda_1)^{-1}G_p^\circ= H_p t(\lambda_2)^{-1} G_p^\circ$$where without loss of generality we assume $\lambda_1\geq \lambda_2\geq 0$. After carrying all the necessary matrix manipulations, the double coset equality implies that the matrix $$\left[\begin{matrix}
            a &bp^{-\lambda_2}\\
            -bDp^{\lambda_1} & ap^{\lambda_1-\lambda_2}
        \end{matrix}\right]$$belongs in $G_p^\circ$. Here $\left[\begin{smallmatrix}
            a & b\\
            -bD & a
        \end{smallmatrix}\right]$ is an element of $H_p$. This implies that both $a$ and $b$ are elements of $\mathbf{Z}_p$. The determinant of the above matrix is equal to $p^{\lambda_1-\lambda_2}(a^2+b^2D)$ and this is an element of $\mathbf{Z}_p^\times$. Since $a,b\in\mathbf{Z}_p$, we conclude that $\lambda_1=\lambda_2$.
    \end{proof}
\end{lem}
\begin{cor}[\Cref{Dns}]\label{lemDns}
    If $H$ is non-split then we have the following double coset decomposition
    $$\mathfrak{G}_p=\bigcup_{\substack{\lambda\in\mathbf{Z}_{\geq 0}\\ a\in\mathbf{Z}}}\mathfrak{H}_p\left(s(\lambda),p^a\right)\mathfrak{G}_p^\circ$$
    where $s(\lambda):=\left[\begin{smallmatrix}
        p^\lambda & \\
         & 1
    \end{smallmatrix}\right]$. Furthermore, the decomposition is disjoint.
    \begin{proof}
        Existence follows at once from \Cref{techdecomp 1}. Disjointness follows again from \Cref{techdecomp 1} and the fact that the pullback of $\mathbf{Z}_p^\times$ under $(v_p\circ\det):H_p\rightarrow \mathbf{Q}_p^\times$, is precisely $H_p^\circ$.
    \end{proof}
\end{cor}
\begin{lem}\label{Murasedecom}
    Let $H $ be split. We have the following double coset decomposition 
    $$G_p=\bigcup_{\lambda\in\mathbf{Z}_{\geq 0}}H_p n_0 s(\lambda) G_p^\circ$$where $n_0:=\left[\begin{smallmatrix}
        1 & 1\\
         & 1
    \end{smallmatrix}\right]$, and $s(\lambda):=\left[\begin{smallmatrix}
        p^\lambda & \\
         & 1
    \end{smallmatrix}\right]$. Furthermore, this decomposition is disjoint.
    \begin{proof}
        This is a special case of Proposition $2.3$ in \cite{murase1996shintani} upon writing $\left[\begin{smallmatrix}
            1 & p^{-\lambda} \\
            & 1
\end{smallmatrix}\right]=s(\lambda)^{-1}n_0s(\lambda).$
    \end{proof}
        \end{lem}
\begin{cor}[\Cref{Ds}]\label{lemDs}
    If $H$ is split then we have the following double coset decomposition
    $$\mathfrak{G}_p=\bigcup_{\substack{\lambda\in\mathbf{Z}_{\geq 0}\\ a,b\in\mathbf{Z}}}\mathfrak{H}_p \left(n_0s(\lambda), \left[\begin{smallmatrix}
        p^a & \\
         & 1
    \end{smallmatrix}\right] p^b\right)\mathfrak{G}_p^\circ.$$
    Furthermore, this decomposition is disjoint.
    \begin{proof}
         The existence of such a decomposition follows at once from \Cref{Murasedecom}. To see disjointness, suppose that we have the following equality of elements of $\mathfrak{G}_p$
        $$\left(n_0s(\lambda), \left[\begin{smallmatrix}
            p^a & \\
             & 1
        \end{smallmatrix}\right]p^b\right)=\left(h n_0s(\lambda_0)k_1, h^{-1}\left[\begin{smallmatrix}
            p^{a_0} & \\
             & 1
        \end{smallmatrix}\right]p^{b_0}k_2\right)$$
        with $h\in H_p$ and $(k_1,k_2)\in \mathfrak{G}_p^\circ$. It follows that $h$ is equal to $\left[\begin{smallmatrix}
            p^{a_0-a} & \\
             & 1
        \end{smallmatrix}\right]p^{b_0-b}k_2$ and thus $$n_0s(\lambda)=\left[\begin{smallmatrix}
            p^{a_0-a} & \\
             & 1
        \end{smallmatrix}\right]p^{b_0-b}k_2n_0s(\lambda_0)k_1.$$
        By \Cref{Murasedecom}, this instantly implies that $\lambda=\lambda_0$. The $(2,1)$-entry of the left hand side is zero, hence the $(2,1)$-entry of $k_1$ must be zero and hence it's $(2,2)$-entry must be a unit. Taking $p$-adic valuation of the $(2,2)$-entry on both sides, we see that $b=b_0$. Finally taking $p$-adic valuation of determinants, yields $a=a_0$. This concludes the proof.
    \end{proof}
\end{cor}
\noindent
We now give a proof for \Cref{lemideal}. We first need a bit of preparation.  
\begin{defn}
    For $k\geq 1$, we define the element $\mathcal{T}(p^k)$ of $\mathcal{H}_{G_p}^\circ$, to be the characteristic function of the set of all elements in $\mathrm{Mat}_2(\mathbf{Z}_p)$, whose determinant generates the ideal $p^k\mathbf{Z}_p$ in $\mathbf{Z}_p$.
    
\end{defn}

\begin{rem}\label{remappend}
    It will be useful for our later computations to have the following explicit formula for the characteristic function $\mathcal{T}(p^k)$.
    $$\mathcal{T}(p^k)=\sum_{\substack{0 \leq  j \leq  i  \leq  k\\ \\ i+j=k}} \ch\left(G_p^\circ \left[\begin{smallmatrix}
        p^i & \\
         &  p^j
    \end{smallmatrix}\right] G_p^\circ\right).$$
    In our notation so far, the operator $\mathcal{T}(p)$ is nothing more than $\mathcal{T}_p$.
\end{rem}

\begin{lem}\label{lem appendix bump}
    For any $k\in\mathbf{Z}_{\geq 2}$, the characteristic function $\mathcal{T}(p^k)$ satisfies the following recurrence relation:
    $$\mathcal{T}(p^k)=\mathcal{T}_p \mathcal{T}(p^{k-1})-p\mathcal{S}_p\mathcal{T}(p^{k-2}).$$
    \begin{proof}
        \cite{bump_1997} Proposition $4.6.4$.
    \end{proof}
\end{lem}

\begin{lem}[\Cref{lemideal}]\label{appendix lemma ideal}
    For any $\lambda\in\mathbf{Z}_{\geq 1}$, the characteristic function $\ch(G_p^\circ s(\lambda) G_p^\circ)$ is an element of the ideal 
    $\left(p+1,\mathcal{T}_p\right)\subseteq\mathcal{H}_{G_p}^\circ(\mathbf{Z}[1/p])$.
    \begin{proof}
        For $\lambda=1$, there is nothing to prove. For $\lambda\geq 2$, we prove the result by induction on $\lambda\geq 2$. We again write $\phi_{s(\lambda)}$ for the function $\ch\left(G_p^\circ\left[\begin{smallmatrix}
            p ^\lambda & \\
             & 1
        \end{smallmatrix}\right]G_p^\circ\right).$
        For $\lambda=2$, we have \begin{align*}
            \phi_{s(2)}&= \mathcal{T}(p^2)-\mathcal{S}_p\\
            &=\mathcal{T}_p^2-(p+1)\mathcal{S}_p
        \end{align*}
        \noindent Now let $\lambda>2$. We have 
        $$\mathcal{T}(p^\lambda)=\phi_{s(\lambda)} + \sum_{\substack{0 <j\leq i< \lambda \\ i+j=\lambda}} \mathcal{S}_p^j\phi_{s(i-j)}$$
        and similar identities by replacing  $\lambda$ with $\lambda-1$, $\lambda-2$. Using \Cref{lem appendix bump}, we have 
        \begin{align}\label{eq:19}
            \nonumber\phi_{s(\lambda)}&=\mathcal{T}_p \mathcal{T}(p^{\lambda-1})-p\mathcal{S}_p \mathcal{T}(p^{\lambda-2})-\sum_{\substack{ 0<j\leq i<\lambda\\ i+j=\lambda}}\mathcal{S}_p^j\phi_{s(i-j)}\\
            &=\left(\mathcal{T}_p\sum_{\substack{ 0\leq j\leq i\leq \lambda-1\\ i+j=\lambda-1}}\mathcal{S}_p^j\phi_{s(i-j)}\right) - \left(p\mathcal{S}_p\sum_{\substack{ 0\leq j\leq i\leq \lambda-2\\ i+j=\lambda-2}}\mathcal{S}_p^j\phi_{s(i-j)}\right)-\left(\sum_{\substack{ 0< j\leq i<\lambda\\ i+j=\lambda}}\mathcal{S}_p^j\phi_{s(i-j)}\right).
        \end{align}
        The first of the three brackets always lies in the ideal in question. If the parity of the integers $\lambda,\lambda-1,\lambda-2$ is odd, even, odd, then every summand in the second and third brackets is of the form $\mathcal{S}_p^j\phi_{s(i-j)}$, with $i-j$ strictly bigger than zero and strictly less than $\lambda$. Hence in this case, we are done by induction and the case $\lambda=1$. If on the other hand the parity of the integers $\lambda, \lambda-1, \lambda-2$ is even, odd, even, then the second bracket contains a summand of the form $p\mathcal{S}_p^{\frac{\lambda}{2}}$
        and the third bracket contains a summand of the form 
        $\mathcal{S}_p^{\frac{\lambda}{2}}.$
        Adding the two together, we end up with
        $  
      (p+1)\mathcal{S}_p^{\frac{\lambda}{2}}
        $
        which is an element of the ideal in question. Finally, every other summand in the second and third brackets of \eqref{eq:19} lies in this ideal by induction, and the case $\lambda=1$. This concludes the proof of the lemma.
    \end{proof}
\end{lem}
\noindent The last thing left to do is to give a proof of \Cref{lemtech}. 
\begin{lem}[\Cref{lemtech}]
   For $\gamma\in H_p^\circ/H_p^\circ(\lambda)$, we write $\gamma s(\lambda)=p^{z_\gamma}\left[\begin{smallmatrix}
        p^{r_\gamma} & \\
        & 1
    \end{smallmatrix}\right]\left[\begin{smallmatrix}
        1 & n_\gamma \\
        & 1
    \end{smallmatrix}\right] k_\gamma\in Z_pA_pN_p G_p^\circ$. There exists a complete set of distinct coset representatives of $H_p^\circ/H_p^\circ(\lambda)$ that can be partitioned as follows:
       \begin{itemize}
          \item There are $\varphi(p^\lambda)$ elements for which $v_p(n_\gamma)=-\lambda$ and $z_\gamma=0$\\
          \item For each $1\leq i\leq \lambda-1$, there  are $\varphi(p^{\lambda-i})$ elements for which $v_p(n_\gamma)=i-\lambda$ and $z_\gamma= i$. Additionally, there are $\varphi(p^{\lambda-i})$ elements for which $v_p(n_\gamma)=i-\lambda$ and $z_\gamma= 0$.\\
          \item Finally, there's a unique element for which $n_\gamma=0$ and $z_\gamma=\lambda$, and there is a unique element for which $n_\gamma=0$ and $z_\gamma=0$.
       \end{itemize}
       \begin{proof}
         Let $\lambda\in\mathbf{Z}_{\geq 0}.$ We have the following complete set of distinct coset representatives for $H_p^\circ/H_p^\circ(\lambda)$
         $$\left\{\left[\begin{smallmatrix}
            a & -D\\
            1 & a
        \end{smallmatrix}\right], \left[\begin{smallmatrix}
            1 & -bD\\
            b & 1
        \end{smallmatrix}\right]\ |\ a\in\mathbf{Z}/p^\lambda\mathbf{Z}\ ,\ b\in p\mathbf{Z}/p^\lambda\mathbf{Z}\right\}.$$
        For $\gamma=\left[\begin{smallmatrix}
            a & -D\\
            1 & a
        \end{smallmatrix}\right]$, we have $\gamma s(\lambda)=\left[\begin{smallmatrix} 
                    ap^\lambda & -D\\
                    p^\lambda & a
                \end{smallmatrix}\right].$ It is not hard to check that Iwasawa decomposition gives
               $$ v_p(n_\gamma)=
                \begin{cases}
                    v_p(a)-\lambda\ ,\ &\text{if }v_p(a)\leq \lambda-1\\
                    \infty \ ,\ &\text{if }a=0
                \end{cases}\ \ \ ,\ z_\gamma=\begin{cases}
                    v_p(a)\ ,\ &\text{if }v_p(a)\leq \lambda-1\\
                    \lambda\ ,\ &\text{if }a=0.
                \end{cases}$$
                If on the other hand $\gamma=\left[\begin{smallmatrix}
            1 & -bD\\
            b & 1
        \end{smallmatrix}\right]$, then $\gamma s(\lambda)=\left[\begin{smallmatrix}
                    p^\lambda & -bD\\
                    bp^\lambda & 1
                \end{smallmatrix}\right]$, and this time Iwasawa decomposition gives 
                $$v_p(n_\gamma)=
                \begin{cases}
                    v_p(a)-\lambda\ ,\ &\text{if }v_p(a)\leq \lambda-1\\
                    \infty \ ,\ &\text{if }a=0
                \end{cases}\ \ ,\ \ z_\gamma=0.$$
                The result now follows from a simple counting argument.
       \end{proof}
\end{lem}
\bibliography{citation} 
\bibliographystyle{alpha}

\noindent\textit{Mathematics Institute, Zeeman Building, University of Warwick, Coventry CV4 7AL,
England}.\\
\textit{Email address}: Alexandros.Groutides@warwick.ac.uk
\end{document}